\documentclass{article}
\usepackage[utf8]{inputenc}
\usepackage{graphicx}
\usepackage{amsfonts}
\usepackage{amsmath}
\usepackage{amssymb}
\usepackage{fancyhdr}
\usepackage{titlesec}
\usepackage{indentfirst}
\usepackage{booktabs}
\usepackage{verbatim}
\usepackage{color}
\usepackage{amsthm}
\usepackage{subfigure}
\usepackage{mathabx}

\usepackage[page,header]{appendix}
\usepackage{titletoc}

\newcommand{\rd}{\,\mathrm{d}}
\numberwithin{equation}{section}
\newtheorem{theorem}{Theorem}[section]
\newtheorem{lemma}[theorem]{Lemma}
\newtheorem{corollary}[theorem]{Corollary}

\newtheorem{proposition}[theorem]{Proposition}

\newtheorem{remark}[theorem]{Remark}

\usepackage{mathtools}

\def\bx{{\bf x}}

\def\cE{\mathcal{E}}
\def\cM{\mathcal{M}}

\def\cW{\mathcal{W}}

\def\cU{\mathcal{U}}
\def\cV{\mathcal{V}}

\def\supp{\textnormal{supp\,}}

\def\pvint{\textnormal{p.v.}\int}
\def\capa{\textnormal{Cap}}

\begin{document}

\title{{A family of nonlocal energies whose minimizers are
supported on two intervals}}

\author{Steven B. Damelin\footnote{zbMATH Open, European Mathematical Society, Department of Mathematics
FIZ Karlsruhe – Leibniz Institute for Information Infrastructure, Franklinstr. 1110587 Berlin, Germany (steve.damelin@gmail.com)} \,\,and Ruiwen Shu\footnote{Department of Mathematics, University of Georgia, Athens, GA 30602 (ruiwen.shu@uga.edu).}}    

\maketitle





\begin{abstract}
In this paper, we consider the one-dimensional interaction energy
$ 
    \frac{1}{2}\int_{\mathbb{R}}(\cW*\rho)(x)\rd{\rho(x)} + \int_{\mathbb{R}}\cU(x)\rd{\rho(x)}
$
where the interaction potential $\cW(x)= -\frac{|x|^b}{b},\,1\le b \le 2$ and the external potential $\cU(x)=\frac{|x|^4}{4}$, and $\rho$ is a compactly supported probability measure on the real line. Our main result shows that the minimizer is supported on two intervals when $1<b<2$, showing in particular how the support of the minimizer transits from an interval (when $b=1$) to two points (when $b=2$) as $b$ increases. As a crucial part of the proof, we develop a new version of the iterated balayage algorithm, the original version of which was designed by Benko, Damelin, Dragnev and Kuijlaars for logarithmic potentials in one dimension. We expect the methodology in this paper can be generalized to study minimizers of interaction energies in $\mathbb{R}^d$ whose support is possibly an annulus.
\end{abstract}

\section{Introduction}

In this paper, we consider the one-dimensional interaction energy
\begin{equation}\label{cE}
    \cE:=\cE[\rho] = \frac{1}{2}\int_{\mathbb{R}}(\cW*\rho)(x)\rd{\rho(x)} + \int_{\mathbb{R}}\cU(x)\rd{\rho(x)}\,,
\end{equation}
where, for $x\in \mathbb R$, the interaction potential $\cW$ and the external potential $\cU$ are given by
\begin{equation}
    \cW(x) = -\frac{|x|^b}{b},\,1\le b \le 2,\quad \cU(x) = \frac{|x|^4}{4}\,,
\end{equation}
and $\rho\in\cM_c(\mathbb{R})$ is a compactly supported probability measure where 
here and throughout, $\cM_c(\mathbb{R})$ denotes the space of compactly supported probability measures on $\mathbb R$.

Here and throughout, $*$ denotes the action of convolution. In order to proceed, we need the motivating:
\begin{proposition}\label{prop_EL}
    There exists a unique minimizer of $\cE$. It is the only element in $\cM_c(\mathbb{R})$ satisfying the Euler-Lagrange conditions
    \begin{equation}\label{EL}
        \cW*\rho+\cU = C_0,\quad\textnormal{on }\supp\rho\,,
    \end{equation}
    and
    \begin{equation}\label{EL2}
        \cW*\rho+\cU \ge C_0,\quad\textnormal{on }\mathbb{R}\,.
    \end{equation}
\end{proposition}
The proof of Proposition \ref{prop_EL} is standard and outlined as below.
\begin{proof}
The argument proceeds as follows: First, for any $R>0$, there exists a minimizer $\rho_R$ in $\cM([-R,R])$ (the space of all probability measures on $[-R,R]$) by a lower-semicontinuity argument as in \cite[Lemma 2.2]{SST15}. 
Using the Euler-Lagrange condition for $\rho_R$ and the fact that $\cU$ grows faster than $-\cW$ at infinity,  there exists $R_*>0$ independent of $R$ such that $\supp\rho_R\subset [-R_*,R_*]$. This implies that $\rho_{R_*}$ is a minimizer in $\cM([-R,R])$ for any $R\ge R_*$, and thus a minimizer in $\cM_c(\mathbb{R})$. The Euler-Lagrange conditions \eqref{EL}\eqref{EL2} for this minimizer can be established easily. 

For $1\le b < 2$, the uniqueness of the minimizer and the sufficiency of the Euler-Lagrange conditions follow from a convexity argument, for example, see the proof of \cite[Theorem 2.1]{DOSW}. Here one needs a crucial inequality that $\int_{\mathbb{R}}(\cW*\mu)(x)\rd{\mu(x)}>0$ for any nontrivial compactly supported mean-zero signed measure $\mu$. This is a consequence of \cite[Theorems 3.10 and 3.8]{Shu_convex}, which gives the Fourier representation of this integral and the positivity of $\hat{\cW}$, respectively. Here, $\hat{\cW}$ denotes the Fourier transform of the function $\cW$.

If $b=2$, then we still have a non-strict inequality $\int_{\mathbb{R}}(\cW*\mu)(x)\rd{\mu(x)}=(\int x\rd{\mu(x)})^2\ge 0$, and equality only holds when $\int x\rd{\mu(x)}=0$. Then we verify that $\rho(x) = \frac{1}{2}(\delta_1(x)+\delta_{-1}(x))$ satisfies \eqref{EL}\eqref{EL2}, with the inequality in \eqref{EL2} being strict on $(\supp\rho)^c$. For any $\rho_1\in\cM_c(\mathbb{R})$ not equal to $\rho$, we define the linear interpolation $\rho_t = (1-t)\rho+t\rho_1$, and we have
\begin{equation}\label{Eineq1}
    \frac{\rd^2}{\rd t^2}\cE[\rho_t] = \int_{\mathbb{R}}(\cW*(\rho_1-\rho))(x)\rd{(\rho_1-\rho)(x)}\ge 0\,,
\end{equation}
and
\begin{equation}\label{Eineq2}
    \frac{\rd}{\rd t}\Big|_{t=0}\cE[\rho_t] = \int_{\mathbb{R}}(\cW*\rho+\cU)(x)\rd{(\rho_1-\rho)(x)} \ge 0\,.
\end{equation}
Here, the inequality \eqref{Eineq2} achieves the equal sign only when $\supp\rho_1\subset \supp\rho=\{\pm1\}$, due to the strict inequality \eqref{EL2} for $\rho$ on $(\supp\rho)^c$. Then it is easy to verify that $\int_{\mathbb{R}}x\rd{(\rho_1-\rho)(x)}\ne 0$ using $\rho_1\ne \rho$, so that the inequality \eqref{Eineq1} is strict in this case. Therefore we see that at least one of the inequalities \eqref{Eineq1} and \eqref{Eineq2} are strict, so we conclude that $\cE[\rho_1]> \cE[\rho]$, i.e., the minimizer of $\cE$ is unique.  The sufficiency of the Euler-Lagrange condition follows by a similar linear interpolation argument.

\end{proof}

The unique minimizer of $\cE$ is known when $b=1,2$. This can be easily checked by verifying the Euler-Lagrange conditions \eqref{EL}\eqref{EL2} explicitly.
\begin{itemize}
    \item When $b=1$, the unique minimizer is 
    \begin{equation}
        \rho(x) = \frac{3}{2}x^2 \chi_{[-1,1]}(x)\,.
    \end{equation}
    Notice that the density touches 0 at $x=0$.
    \item When $b=2$, the unique minimizer is 
    \begin{equation}
        \rho(x) = \frac{1}{2}(\delta_1(x)+\delta_{-1}(x))\,.
    \end{equation}
\end{itemize}
In this paper we give a qualitative description of the minimizer when $1<b<2$, and we give an understanding of how the support of the minimizer transits from an interval to two points as $b$ increases. Our main theorem is the following:
\begin{theorem}\label{thm_main}
    Assume $1<b<2$. Then the unique minimizer $\rho$ of $\cE$ satisfies 
    \begin{equation}
        \supp\rho = [-R_2,-R_1]\cup [R_1,R_2]\,,
    \end{equation}
    for some $R_2>R_1>0$ depending on $b$. $\rho$ is a locally integrable function, satisfying that
    \begin{equation}
        \rho(x)\cdot \big((R_2^2-x^2)(x^2-R_1^2)\big)^{\frac{b-1}{2}}
    \end{equation}
    is $C^1$ (continuously differentiable) and strictly positive on $\supp\rho$.
\end{theorem}

The minimizer of $\cE$ with $b=1.3$ is illustrated in Figure \ref{fig1}.

\begin{figure}[htp!]
 	\begin{center}
 		\includegraphics[width=0.7\textwidth]{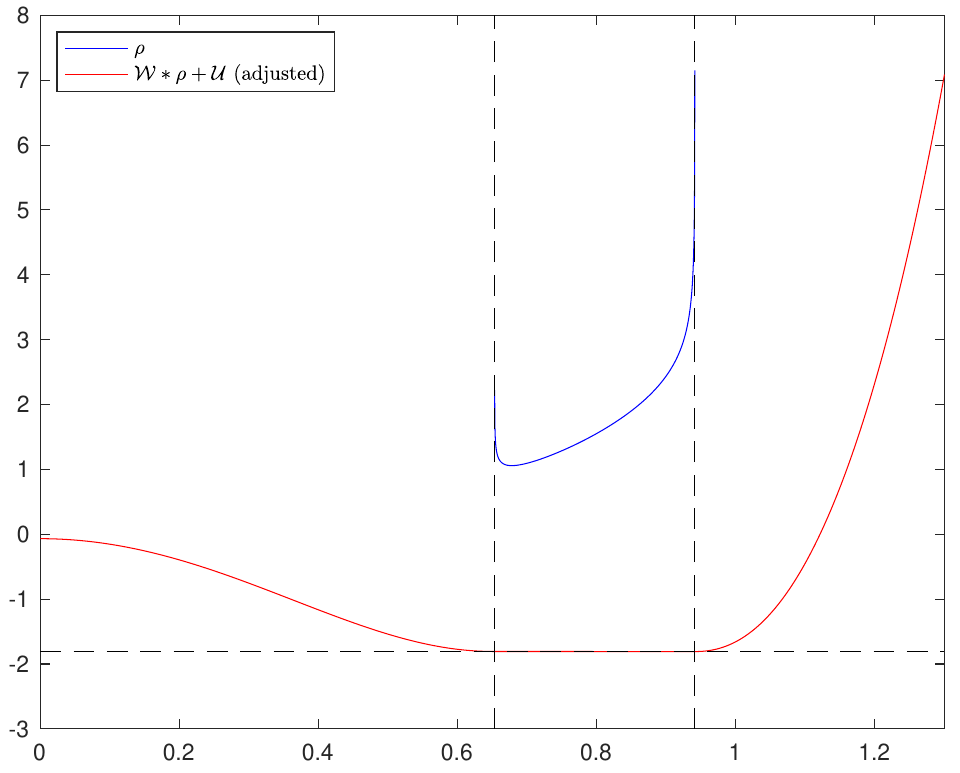}
 		
 		\caption{The minimizer $\rho$ of $\cE$ with $b=1.3$, and its generated potential $\cW*\rho+\cU$. We have $R_1\approx 0.6532$ and $R_2\approx 0.9411$. We only sketch the graphs on $x>0$ due to symmetry. We sketch the adjusted version $50(\cW*\rho+\cU+0.6)$ of the generated potential for clearer illustration.}
 		\label{fig1}
 	\end{center}	
 \end{figure}

Throughout, $C, C_1, ..$ will always denote absolute positive constants and the same symbol may denote a different constant from time to time. By the symbol $\|\cdot\|$, we mean a norm and its context will be clear.

\subsection{Motivation of this paper}

The main motivation of this paper is to develop the necessary tools to study the minimizers of interaction energies in $\mathbb{R}^d$ whose support is possibly an annulus $\{\bx\in\mathbb{R}^d:R_1\le |\bx|\le R_2\},\,R_2>R_1>0$.

The interaction energy $\cE[\rho]$ in multi-dimensions with various types of $\cW$ and $\cU$ appears naturally in the study of collective dynamics in biological and social sciences. 
Its associated Wasserstein-2 gradient flow is the aggregation equation, and thus the global/local minimizers of $\cE$ are expected to be the stable steady states for the aggregation equation. We refer to \cite{CS21} and the references therein for a review of the literature on this interaction energy. Interactive energy with external fields for logarithmic kernels appear in many areas for example integrable systems, orthogonal polynomials and approximation theory \cite{Deift,LL,Lub,Mha,Lub93}. 

For the attractive-repulsive interaction potential $\cW_{a,b}(\bx)=\frac{|\bx|^a}{a}-\frac{|\bx|^b}{b}$ (without external potential) on $\mathbb{R}^d$, if $-d<b\le 2 \le a \le 4,\,a>b,\,(a,b)\ne (4,2)$, it is known that there exists a minimizer \cite{SST15,CCP15}, and the minimizer is unique up to translation and radially symmetric \cite{Lop19,CS21,Fra}. If $d\ge 2$ and one further has $b\ge b_*(a):=\frac{-10+3a+7d-a d-d^2}{d+a-3}$, then the minimizer is the uniform measure on a sphere \cite{BCLR13_2,DLM1,FM}. If one instead has $b\le 2-d$, then it is easy to see that the support of the minimizer is a ball \cite[Proposition 2.1]{Shu_expli}. If one fixes $a\in [2,4]$ and views $\cW_{a,b}$ as a one-parameter family of $b$, it is natural to ask how the support of the minimizer changes from a ball to a sphere when $b$ increases.

If $a=2$, then the minimizer is a constant multiple of $(R_2^2-|\bx|^2)^{\frac{d-b+2}{2}}$ for $-d<b<\min\{4-d,2\}$ for some $R_2=R_2(d,b)>0$ \cite{Frost,CV11,CH17,CS21,Fra}. Therefore, for $d\ge 3$ and $a=2$, the support of the minimizer suddenly changes from a ball to a sphere at $b=b_*(2)=4-d$. However, if $a=4$, the explicit formula of the minimizer is given by $(R_4^2-|\bx|^2)^{\frac{d-b+2}{2}}(A_1 R_4^2 + A_2 (R_4^2-|\bx|^2))$ for $-d<b\le \bar{b} := \frac{2+2d-d^2}{d+1}$ (where $R_4>0,A_1>0,A_2\in\mathbb{R}$ depends on $d$ and $b$), which is supported on a ball \cite{CH17,CS21} \footnote{Here, \cite{CH17} gives the explicit formula as a steady state. \cite{CS21} proves that it is the minimizer for $2-d<b\le \bar{b}$. The case $-d<b\le 2-d$ can be treated via \cite[Proposition 2.1]{Shu_expli}.}. Its density touches zero smoothly at the origin when $b=\bar{b}$. Since $\bar{b}<b_*(4)$, it is expected that a `hole' appears near the origin when $b>\bar{b}$, and the hole gets larger and larger as $b$ increases, until the support becomes a sphere at $b_*(4)$. In particular, one may conjecture that the support is an annulus when $\bar{b}<b<b_*(4)$. The local stability analysis in \cite{BCLR13_2} for a sphere-shaped steady state also suggests that `fattening' the sphere decreases the energy when $b$ is slightly less than $b_*(a)$, suggesting the appearance of an annulus-shaped minimizer.

To understand the shape of the support, especially in the case $2-d<b<\min\{4-d,2\}$, one known method \cite{CS21} is to study the sign of $\Delta^2\cW*\rho$ in vacuum regions. Suppose $\Delta^2\cW(\bx)<0,\,\forall\bx\ne0$, then, at least for radial minimizers, one can guarantee that the support is ball. This is obtained by proof by contradiction, via a maximum principle argument in a possible bounded vacuum region. 

A recent work \cite{CMSVW} studies a similar family of energy minimizers with $\cW(\bx)=-\frac{|\bx|^b}{b}$ and $\cU(\bx)=\frac{|\bx|^a}{a}$, and it shows that the minimizer is the uniform measure on a sphere if $a$ is greater than or equal to some threshold value $\alpha_{b,d}$, see (1.6) therein. Using a perturbative method with the sign control of $\Delta^2\cW*\rho$ near a sphere, an upcoming work of the second author with Carrillo, Huang and Saff \cite{CHSS} shows that the support of the minimizer is an annulus when $a$ is slightly less than $\alpha_{b,d}$.

However, we do not believe that analyzing the sign of $\Delta^2\cW*\rho$ can justify the full transition of the support from a ball to a sphere. In fact, in a non-perturbative setting, we can only hope to use this method if $\Delta^2\cW(\bx)<0,\,\forall\bx\ne0$. But this cannot be the case if the support of the minimizer is annulus shaped, due to the the vacuum region near the origin.

The energy $\cE$ in \eqref{cE} considered in this paper is a one-dimensional problem, where the support of the minimizer is a ball (an interval) when $b=1$ and a sphere (two points) when $b=2$. Our main result, Theorem \ref{thm_main}, shows that the minimizer is supported on an annulus (two intervals) when $1<b<2$. Therefore, in this one-dimensional model, we justify the transition of the minimizer along a one-parameter family from ball-shaped to sphere-shaped through annuli. 

The main approach of this paper is the \emph{iterated balayage algorithm} (IBA), which constructs a minimizer or steady state via an iteration that starts from a \emph{signed} steady state. By keeping track of certain monotonicity along the iteration, one can get a nice understanding of the shape of the minimizer.

In the following two subsections we review the literature on the IBA for the logarithm potential $W_{\log}(x)=-\ln|x|$. Then in Section \ref{sec_sketch} we sketch the proof of our main result, Theorem \ref{thm_main}.

\subsection{Logarithmic balayage onto a finite number of intervals}


As a brief review of the one-dimensional potential theory for the logarithmic potential, we consider $[-1,1]$ as the underlying space. The (logarithmic) equilibrium measure in the presence of a continuous 
external field $U: [-1,1] \longrightarrow {\bf R}$ is 
the unique $\mu\in \cM([-1,1])$ 
satisfying for some constant $C_0$,
\begin{equation} \label{eq11}
	\left\{ \begin{array}{rclcl}
	(W_{\log}*\mu)(x) + U(x) & = & C_0, && x \in \supp\mu, \\
	(W_{\log}*\mu)(x) + U(x) & \geq & C_0, && x \in [-1,1].
    \end{array} \right.
\end{equation}


We recall the notion of balayage onto a compact set,
see \cite{Landkof}. 
Let $K$ be a compact subset of the 
complex plane with positive logarithmic capacity
and such that the complement $\mathbb{R} 
\setminus K$ is regular (cf. \cite{StahlTotik}) for the Dirichlet
problem. In particular, $K$ can be a union of finitely many intervals. Then, if $\nu$ is any finite positive Borel measure 
on  $\mathbb{R}$ with compact support, there exists
a unique positive measure $\hat{\nu}$ supported on $K$
such that $\| \nu \| = \|\hat{\nu} \|$,  and for
some constant $C$,
\begin{equation} \label{eq31}
	(W_{\log}*\hat{\nu})(z) = (W_{\log}*\nu)(z) + C,
	\qquad z \in K.
\end{equation}
The measure $\hat{\nu}$ is called the (logarithmic) balayage of $\nu$
onto $K$ and we denote it by $\widetilde{Bal}(\nu,K)$. For a
signed measure $\sigma = \sigma^+ - \sigma^-$, we define 
$\widetilde{Bal}(\sigma,K) :=  \widetilde{Bal}(\sigma^+, K) - \widetilde{Bal}(\sigma^-, K)$.

\subsection{The iterated balayage algorithm (IBA)}

The iterated balayage algorithm first appeared in \cite{KD99}, and has also been used in numerous papers, for example, \cite{DK99,DDK01,BDD06,BD12,DOSW25}. It gives an iterative 
method to solve an equilibrium problem 
with an external field. 

Given a smooth enough external field $U$ on $[-1,1]$ one proceeds 
as follows.
Suppose one knows that the support of $\mu$ is contained
in the interval $[a,b]\subset[-1,1]$. (For example, one could take
$[-1,1]$, but it will be useful to have some freedom here.) 
Then the first step is to solve the integral equation
\begin{equation} \label{eq41}
	(W_{\log}*\sigma_0)(x) = -U(x) + C_0,
	\qquad a < x < b, 
\end{equation}
to find a function $\sigma_0$ supported on $[a,b]$, subject to the total mass condition
\begin{equation} \label{eq42}
	 \int_a^b \sigma_0(t) dt = 1. 
\end{equation}
Formally differentiating (\ref{eq41}) with respect to $x$, 
one obtains the singular integral equation
\begin{equation} \label{eq43}
	\pvint_a^b \frac{\sigma_0(t)}{x-t} dt = U'(x), 
	\qquad a < x < b.
\end{equation}
Here $\pvint$ is used to denote a Cauchy principle value
integral.

It is well-known that if $U$ is smooth enough, for example
if $U$ is differentiable with a H\"older continuous derivative,
i.e., $U \in C^{1+\epsilon}([a,b])$ for some $\epsilon > 0$, then
(\ref{eq42})--(\ref{eq43}) has the unique solution
\begin{equation} \label{eq44}
	\sigma_0(t) = \frac{1}{\pi \sqrt{(b-t)(t-a)}} 
	\left[1 + \frac{1}{\pi} \pvint_{a}^b 
	\frac{U'(s)}{s-t} \sqrt{(b-s)(s-a)} ds \right],
	\quad a < t < b,
\end{equation}
where the above integral is again a Cauchy principle value
integral, see \S 42.3 of \cite{Gak}.
If the function $\sigma_0$ happens to be non-negative on
$[a,b]$ then it is the density of the equilibrium measure 
with external field $U$ and we are done.
If not, then let $\sigma_0 = \sigma_0^+ - \sigma_0^-$ be the Jordan
decomposition of $\sigma_0$ and 
\[ \Sigma_1 := \supp(\sigma_0^+). \]
It was shown in \cite{KD99} that $\mu \leq \sigma_0^+$
and $\supp\mu \subset \Sigma_1$, so that in
determining $\mu$ and its support we may restrict ourselves to
$\Sigma_1$. The next step is to consider the integral equation 
on $\Sigma_1$
\begin{equation} \label{eq45}
	(W_{\log}*\sigma_1)(x) = -U(x) + C_1,
	\qquad x \in \Sigma_1 \,,
\end{equation}
to find a function $\sigma_1$ supported on $\Sigma_1$, subject to the condition
\begin{equation} \label{eq46}
	\int_{\Sigma_1} \sigma_1(t)\rd{t} = 1.
\end{equation}
The solution to (\ref{eq45})-(\ref{eq46}) can be expressed in 
terms of the measure $\sigma_0$: it is simply the balayage 
of $\sigma_0$ onto $\Sigma_1$.

To describe this process, an operator $J$ was introduced in 
\cite{KD99} on all finite signed measures $\sigma$ on $[-1,1]$ 
with $\int d\sigma = 1$ and $\capa(\supp(\sigma^+)) > 0$ ($\capa$ means the logarithmic capacity) as 
follows
\[ J(\sigma) := \sigma^+ - \widetilde{Bal}(\sigma^-, \supp(\sigma^+))
	= \widetilde{Bal}(\sigma, \supp(\sigma^+)). \]
The operator $J$ sweeps the negative part of the measure
$\sigma$ onto the support of the positive part, so that
in particular $J(\sigma)^+ \leq \sigma^+$.

Returning to (\ref{eq45})-(\ref{eq46}), we see that 
$\sigma_1$ is given by
\[ \sigma_1 = J(\sigma_0). \]
Continuing in this way, we write for every $k \geq 1$,
\begin{equation} \label{eq47}
	\Sigma_k := \supp(\sigma_{k-1}^+),
	\qquad \sigma_k := J(\sigma_{k-1}) =
	J^k(\sigma_0). 
\end{equation}
The measures $\sigma_k$ are signed measures which have
a Jordan decomposition $\sigma_k = \sigma_k^+ - \sigma_k^-$.
It follows as in \cite{KD99} that 
\begin{equation} \label{eq48}
	\sigma_0^+ \geq \sigma_1^+ \geq \cdots  \geq \mu, 
\end{equation}
and
\begin{equation} \label{eq49}
	[a,b] \supset \Sigma_1 \supset \Sigma_2 
	\supset \cdots \supset \supp(\mu). 
\end{equation}

Under quite general conditions one expects from (\ref{eq48})
that the sequence $\{\sigma_k^+ \}_{k=0}^{\infty}$ converges
in weak$^*$ sense
to the equilibrium measure $\mu$ (in which case we say that the IBA converges). In fact, it suffices to show that the negative parts $\sigma_k^-$
tend to zero as $k$ tends to $\infty$, which is true if one can control the limiting
set $\Sigma^*=\lim_{k\rightarrow\infty}\Sigma_k$. To do this, one can usually start by showing that for every $k$, $\Sigma_{k+1}$ 
is contained in the interior of $\Sigma_k$ with respect 
to $[-1,1]$. Then it follows that any possible limit measure of a subsequence of 
$\{ \sigma_k^- \}$ is supported on $\partial \Sigma^*$,
the topological boundary of $\Sigma^*$ in $[a,b]$. If the capacity
of $\partial \Sigma^*$ is zero (for example, if $\Sigma^*$ is a union of finitely many intervals), then $\partial \Sigma^*$
cannot support a non-zero measure with a finite potential. 
This forces the sequence $\{ \sigma_k^- \}$
to converge to $0$, which proves the convergence of the
IBA.

Besides presenting a possible algorithm for numerical 
calculations, the IBA can also
be used to prove rigorous results on the support of 
$\mu$ in certain situations.
This was done in \cite{KD99}, where for the class of
external fields there it was shown that
$\Sigma_k$ consists of at most two intervals for every
$k$, which led to the result that the support of $\mu$ also consists of at most two intervals. Moreoever, in \cite{DK99}, Damelin and Kuijlaars studied the support of the equilibrium measure for a logarithmic potential with an external field of the form $-c x^{2m+1},\,x\in [-1,1]$ with $c>0$ and $m$ a positive integer. The authors showed that the support of the equilibrium measure consists of at most two intervals which resolved a question of Deift, Kriecherbauer and McLaughlin \cite{DKM98}. In \cite{DDK01}, the first author, Dragnev and Kuijlaars, investigated the support of the equilibrium measure for a logarithmic potential with a class of non-convex, non-smooth external fields on a finite interval. More precisely, the authors obtained a sufficient condition which ensures that the support of the equilibrium measure consists of at most two intervals. This condition was then applied to external fields of the form $-c{\rm sign}(x)|x|^{\alpha}$ with $c>0$, $\alpha\geq 1$ and $x\in [-1,1]$. In \cite{BDD06}, the first author, Benko and Dragnev studied the support of the equilibrium measure for a logarithmic potential with a class of external fields defined on arcs of the unit circle and on intervals of the compactified real line. Several sufficient conditions were given to ensure that the  support of the equilibrium measure is one interval or one arc.
\medskip

The recent paper \cite{DOSW25} uses the iterated balayage algorithm for the Riesz kernel to prove that certain equilibrium measures on $[-1,1]$ with external fields are supported on two intervals of the form $[-1,R]\cup[R,1]$ for some $R>0$.

\subsection{Sketch of proof and outline of the paper}\label{sec_sketch}


We start by noticing that taking the negative Laplacian on the Euler-Lagrange condition \eqref{EL} leads to a similar equation \eqref{WUEL} with \eqref{WU}. Here, the new interaction potential $W(x)=|x|^{-s},\,0<s<1$ is a Riesz kernel that is allowed for the IBA. To seek for an annulus-shaped probability measure satisfying the Euler-Lagrange condition \eqref{EL}, we can reverse this procedure, and seek for certain steady state associated to $W,U$, as described in Lemma \ref{lem_reduct}. Here we have two undetermined coefficients $R_1,R_2$, and we can eliminate one by a rescaling technique as described in Lemma \ref{lem_lam2}. At the end, we only need to study the one-parameter family $\mu_\lambda$, given by Lemma \ref{lem_lam1}, and find one of them satisfying $F(\mu_\lambda)=0$ (cf. \eqref{F}).

In Section 3 we review some basic results from the literature about equilibrium measures and the balayage operator, as well as give some extensions. Then in Sections \ref{sec_mulam} and \ref{sec_mulam2} we study the family of signed measures $\mu_\lambda$, giving their existence, uniqueness, regularity and some continuity and limiting behavior. This allows us to use the IBA to study the monotone properties of $\mu_\lambda$ in Section \ref{sec_iba}. To be precise, in Proposition \ref{prop_iba} we construct an increasing sequence $\{\lambda_k\}$ so that $\{\mu_{\lambda_k}\}$ are related iteratively through the balayage operator, analogous to the sequence $\{\sigma_k\}$ in \eqref{eq47}. Then we find a value $\lambda_\infty=\lim_{k\rightarrow\infty}\lambda_k$ which is the smallest $\lambda$ value such that $\mu_\lambda$ is positive, and we guarantee that $F[\mu_{\lambda_\infty}]>0$. This allows us to find the correct value $\lambda_*\in (\lambda_\infty,1)$ satisfying the assumptions Lemma \ref{lem_lam2}, and thus to finish the proof of Theorem \ref{thm_main} in Section \ref{sec_main}.

We remark that the IBA in our paper shares some common features as the one in \cite[Section 6]{DOSW25}, but also has significant differences. This will be further explained in Remark \ref{rmk_compare}.

\section{Reduction to steady states on annuli}

For the minimizer $\rho$ of $\cE$, suppose $K:=\supp\rho$ is a finite union of closed intervals, then in the interior of each interval we may take $-\frac{1}{1-s}\cdot\frac{\rd^2}{\rd x^2}$ of the Euler-Lagrange condition \eqref{EL} and get
\begin{equation}\label{WUEL}
    W*\rho+U = 0\,,
\end{equation}
in the interior of $K$, where
\begin{equation}\label{WU}
    W(x) = -\frac{1}{1-s}\cW''(x) = |x|^{-s},\,s = 2-b \in (0,1),\quad U(x) = -\frac{1}{1-s}\cU''(x) = -\frac{3}{1-s}x^2\,.
\end{equation}
This implies that $\rho$ is the minimizer of 
\begin{equation}\label{E}
    E[\rho] = \frac{1}{2}\int_{K}(W*\rho)(x)\rd{\rho(x)} + \int_{K}U(x)\rd{\rho(x)}\,,
\end{equation}
in $\cM(K)$ (cf. the proof of Proposition \ref{prop_EL}, which also works if $(W,U)$ is in place of $(\cW,\cU)$).

We will partly reverse this procedure, and find the minimizer of $\cE$ as a minimizer of $E$ on some annulus $K$. For $R_2>R_1>0$, denote
\begin{equation}
    K_{R_1,R_2} = [-R_2,-R_1]\cup [R_1,R_2]\,.
\end{equation}
Also denote
\begin{equation}
    V[\rho] = W*\rho+U\,.
\end{equation}

We need:
\begin{lemma}\label{lem_reduct}
    Assume $1<b<2$ and denote $s,W,U,V[\rho],K_{R_1,R_2}$ as above. For given $R_2>R_1>0$, suppose $\rho\in\cM(K_{R_1,R_2})$ satisfies that $V[\rho]$ is continuous on $\mathbb{R}$, with
    \begin{equation}\label{lem_reduct_1}
        V[\rho] = 0,\quad \textnormal{on }K_{R_1,R_2}\,,
    \end{equation}
    and
    \begin{equation}\label{lem_reduct_2}
        \int_{-R_1}^{R_1} V[\rho](x)\rd{x} = 0\,.
    \end{equation}
    Then $\rho$ satisfies \eqref{EL}.
\end{lemma}

As a consequence, suppose one can find $R_1,R_2$ and $\rho\in\cM(K_{R_1,R_2})$ satisfying \eqref{lem_reduct_1} and \eqref{lem_reduct_2}, then $\rho$ is the minimizer of $\cE$ provided that one can also verify \eqref{EL2}.

\begin{proof}
First notice that \eqref{lem_reduct_1} implies that $\rho$ is the unique minimizer of $E$ in $\cM(K_{R_1,R_2})$, and thus it is even. Therefore the function $\cV[\rho]:=\cW*\rho+\cU$ is also even. Since $-\cV''[\rho] = (1-s)V[\rho] = 0$ on $K_{R_1,R_2}$, we see that $\cV[\rho]$ is a linear function on $[-R_2,-R_1]$ and $[R_1,R_2]$, i.e.,
\begin{equation}
    \cV[\rho](x) = C_0 + C_1 x,\quad\textnormal{on }[R_1,R_2];\qquad \cV[\rho](x) = C_0 - C_1 x,\quad\textnormal{on }[-R_2,-R_1]\,,
\end{equation}
for some constants $C_0,C_1\in\mathbb{R}$.

Since $V[\rho]$ is continuous on $\mathbb{R}$, we see that $\cV[\rho]\in C^2(\mathbb{R})$. Therefore we may calculate
\begin{equation}
    \cV'[\rho](R_1) - \cV'[\rho](-R_1) = \int_{-R_1}^{R_1} \cV''[\rho](x)\rd{x} = -(1-s)\int_{-R_1}^{R_1} V[\rho](x)\rd{x} = 0\,,
\end{equation}
by \eqref{lem_reduct_2}. This implies $C_1=0$, and thus \eqref{EL} is obtained.

\end{proof}

Then, in the following two lemmas, we give a rescaling argument which reduces the two parameters $R_1,R_2$ into one parameter $\lambda\in (0,1)$.
\begin{lemma}\label{lem_lam1}
    Let $\lambda\in [0,1)$. Then there exists a unique signed measure $\mu_\lambda$ supported on $K_{\lambda,1}$ such that 
    \begin{equation}\label{lem_lam1_1}
        V[\mu_\lambda] = 0,\quad \textnormal{on }K_{\lambda,1}\,.
    \end{equation}
    $\mu_\lambda$ satisfies that $V[\mu_\lambda]$ is continuous on $\mathbb{R}$.
\end{lemma}
See Figure \ref{fig2} for an illustration of $\mu_\lambda$ for various values of $s$ and $\lambda$.

\begin{figure}[htp!]
 	\begin{center}
 		\includegraphics[width=0.99\textwidth]{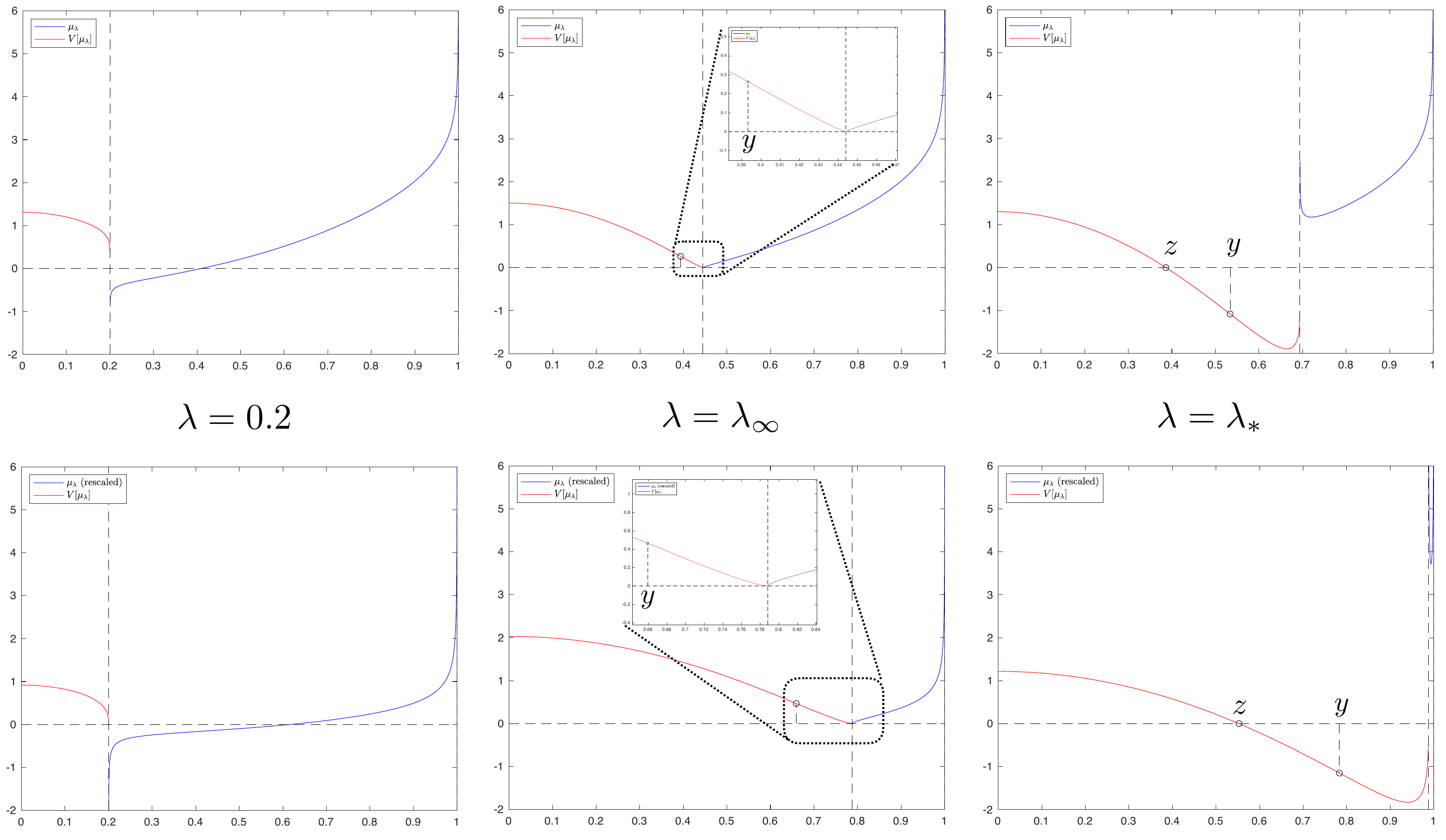}
 		
 		\caption{The signed measure $\mu_\lambda$ and the corresponding $V[\mu_\lambda]$. Upper row: $s=0.7$ (corresponding to the case in Figure \ref{fig1}), with $\lambda=0.2$, $\lambda=\lambda_\infty\approx 0.4440$ and $\lambda=\lambda_*\approx 0.6941$ (see Sections \ref{sec_iba} and \ref{sec_main} for the definition of $\lambda_\infty$ and $\lambda_*$). Lower row: $s=0.3$, with $\lambda=0.2$, $\lambda=\lambda_\infty\approx 0.7880$ and $\lambda=\lambda_*\approx 0.9876$. For this row, the graph of $\frac{1}{10}\mu_\lambda$ is sketched for better illustration. For $\lambda=\lambda_\infty$, the number $y$ in the proof of Item 4 of Proposition \ref{prop_iba} is marked, and the graphs near $x=\lambda$ is zoomed in to show the regularity near this point. For $\lambda=\lambda_*$, the numbers $y$ and $z$ in Section \ref{sec_main} are marked.}
 		\label{fig2}
 	\end{center}	
 \end{figure}

The proof of Lemma \ref{lem_lam1}, which is based on the balayage operator and an explicit construction of $\mu_0$, will be delayed to Section \ref{sec_mulam}.

For $0<\lambda<1$, denote
\begin{equation}\label{F}
    F(\lambda):= \int_{-\lambda}^\lambda V[\mu_\lambda](x)\rd{x}\,.
\end{equation}

\begin{lemma}\label{lem_lam2}
    Let $\lambda\in (0,1)$, $R_2>0$, and $\mu_\lambda$ as given in Lemma \ref{lem_lam1}. Denote $R_1=\lambda R_2$, and
    \begin{equation}
        \mu_{R_1,R_2}(x) = R_2^{1+s}\mu_{\lambda}\big(\frac{x}{R_2}\big)\,,
    \end{equation}
    as a signed measure supported on $K_{R_1,R_2}$. Then
    \begin{equation}\label{lem_lam2_1}
        V[\mu_{R_1,R_2}] = 0,\quad \textnormal{on }K_{R_1,R_2}\,,
    \end{equation}
    and
    \begin{equation}\label{lem_lam2_2}
        \int_{-R_1}^{R_1} V[\mu_{R_1,R_2}](x)\rd{x} = R_2^3 F(\lambda)\,.
    \end{equation}

    As a consequence, if $\mu_\lambda$ is positive and $F(\lambda)=0$, then with $R_2=\Big(\int_{K_{\lambda,1}}\mu_{\lambda}(x)\rd{x}\Big)^{-1/(2+s)}$, the assumptions of Lemma \ref{lem_reduct} are satisfied with $\rho=\mu_{R_1,R_2}$.
\end{lemma}

\begin{proof}
We first calculate
\begin{equation}\label{VR2calc}\begin{split}
    V[\mu_{R_1,R_2}](x) = & R_2^{1+s}\int_{K_{R_1,R_2}}|x-y|^{-s}\mu_{\lambda}\big(\frac{y}{R_2}\big)\rd{y} -\frac{3}{1-s}x^2 \\
    = & R_2^{2+s}\int_{K_{\lambda,1}}|x-R_2 y|^{-s}\mu_{\lambda}(y)\rd{y} -\frac{3}{1-s}x^2 \\
    = & R_2^{2}\int_{K_{\lambda,1}}\big|\frac{x}{R_2}- y\big|^{-s}\mu_{\lambda}(y)\rd{y} -\frac{3}{1-s}x^2 \\
    = & R_2^{2} V[\mu_\lambda]\big(\frac{x}{R_2}\big) \,.
\end{split}\end{equation}
If $x\in K_{R_1,R_2}$, then the last quantity is zero by \eqref{lem_lam1_1}. This gives \eqref{lem_lam2_1}. 

To get \eqref{lem_lam2_2}, we calculate
\begin{equation}
    \int_{-R_1}^{R_1} V[\mu_{R_1,R_2}](x)\rd{x} = R_2^{2}\int_{-R_1}^{R_1} V[\mu_\lambda]\big(\frac{x}{R_2}\big)\rd{x} = R_2^{3}\int_{-\lambda}^{\lambda} V[\mu_\lambda](x)\rd{x} = R_2^3 F(\lambda)\,.
\end{equation}

Further assume that $\mu_\lambda$ is positive and $F(\lambda)=0$. We calculate
\begin{equation}
    \int_{K_{R_1,R_2}}\mu_{R_1,R_2}(x)\rd{x} = R_2^{1+s}\int_{K_{R_1,R_2}}\mu_{\lambda}\big(\frac{x}{R_2}\big)\rd{x} = R_2^{2+s}\int_{K_{\lambda,1}}\mu_{\lambda}(x)\rd{x}  \,.
\end{equation}
Therefore, by choosing $R_2 = \Big(\int_{K_{1,\lambda}}\mu_{\lambda}(x)\rd{x}\Big)^{-1/(2+s)}$, we see that $\mu_{R_1,R_2}$ is nonnegative with $\int_{K_{R_1,R_2}}\mu_{R_1,R_2}(x)\rd{x}=1$, i.e., $\mu_{R_1,R_2}\in \cM(K_{R_1,R_2})$. The continuity of $V[\mu_{R_1,R_2}]$ on $\mathbb{R}$ follows from that of $V[\mu_\lambda]$; \eqref{lem_reduct_1} follows from \eqref{lem_lam2_1}; \eqref{lem_reduct_2} follows from \eqref{lem_lam2_2} and $F(\lambda)=0$.

\end{proof}

\section{{Some known results and their extensions}}\label{sec_prep}

Within this section, we consider $0<s<1$ and $W(x)=|x|^{-s}$. We say a set $K\subset\mathbb{R}$ has \emph{positive capacity} if $\inf_{\rho\in\cM(K)}E_W[\rho]<\infty$, where $E_W[\rho]=\frac{1}{2}\int_{\mathbb{R}}(W*\rho)(x)\rd{\rho(x)}$, otherwise we say $K$ has \emph{zero capacity}. We say a condition is satisfied \emph{quasi-everywhere (q.e.)} on $K$ if the failure set has zero capacity.

\subsection{Basic results from the literature}

We first recall some known results from the literature, which can be found in \cite{DOSW25} (some of these are classical, and appeared in earlier references like \cite{Landkof,Wallin} as cited in \cite{DOSW25}). Many of these results can be formulated in $\mathbb{R}^d$, but we only use them in one dimension for our purposes.

\begin{proposition}\label{prop_equi}
    Let $K\subset\mathbb{R}$ be any compact set with positive capacity. Then there exists a unique element $\rho$ in $\cM(K)$, called the \emph{equilibrium measure} on $K$, satisfying the condition
    \begin{equation}\label{prop_equi_1}
        W*\rho = C_0,\quad\textnormal{q.e. }K\,,
    \end{equation}
    for some $C_0\in\mathbb{R}$. It is the unique minimizer of the energy $\frac{1}{2}\int_K (W*\rho)\rd{\rho}$ in $\cM(K)$, and it is the unique element in $\cM(K)$ satisfying \eqref{prop_equi_1}. The restriction of $\rho$ in the interior of $K$ is a real analytic function.
    
    If $K$ is a finite union of disjoint closed intervals, then near any endpoint $x_0$ of these intervals, the equilibrium measure $\rho$ behaves like $M|x-x_0|^{-\frac{1-s}{2}}$ for some $M>0$ depending on $K$ and $x_0$. As a consequence, $W*\rho$ is continuous on $\mathbb{R}$, and thus $W*\rho=C_0$ on $K$.

    The equilibrium measure on an interval $[a,b]$ is given by
    \begin{equation}
        \omega_{a,b}(x) = (b-a)^{-s}\frac{\Gamma(1+s)}{\Gamma(\frac{1+s}{2})^2}\cdot \big((b-x)(x-a)\big)_+^{-\frac{1-s}{2}}\,.
    \end{equation}
\end{proposition}

Here, $\Gamma$ is the usual Gamma function. This proposition is essentially \cite[Theorem 2.1]{DOSW25}, where the formula of $\omega_{a,b}$ is \cite[eq. (2.5)]{DOSW25}. The continuity of $W*\rho$ is a direct consequence of the boundary behavior and inner regularity of $\rho$.

\begin{proposition}\label{prop_baldef}
    Let $K\subset\mathbb{R}$ be any compact set with positive capacity, and $\mu$ be a positive measure on $\mathbb{R}$ with finite total mass. Then there exists a unique positive measure $\nu$ supported on $K$ satisfying the property    \begin{equation}\label{prop_baldef_1}
        (W*\nu)(x) = (W*\mu)(x),\quad \textnormal{q.e. }K\,.
    \end{equation}
    $\nu$ is denoted as $Bal(\mu,K)$. One has $\|Bal(\mu,K)\| \le \|\mu\|$ for any positive measure $\mu$ with finite total mass.

    If $K$ is a finite union of disjoint closed intervals and $y\notin K$, then  $W*Bal(\delta_y,K)$ is continuous on $\mathbb{R}$, and thus $W*Bal(\delta_y,K)=W*\delta_y$ on $K$.

    Let $I=[a,b]$ be a closed interval and $y\notin I$. Then 
    \begin{equation}\label{prop_baldef_2}
        Bal(\delta_y,[a,b])(x) = \frac{\cos\frac{\pi s}{2}}{\pi}\big((b-y)(a-y)\big)^{\frac{1-s}{2}} \cdot \big((b-x)(x-a)\big)_+^{-\frac{1-s}{2}} |x-y|^{-1}\,.
    \end{equation}
\end{proposition}
This proposition comes from \cite[Propositions 2.3, 2.6]{DOSW25} and a summary of the paragraphs above it. We remark that this version of balayage is slightly different from the operator $\widetilde{Bal}$ introduced in \eqref{eq31}. In fact, $\widetilde{Bal}$ keeps the total mass of the measure but allows the potential to shift by a constant $C$; $Bal$ keeps the potential the same (cf. \eqref{prop_baldef_1}) but allows the total mass to change.

For a signed measure $\mu$ with $||\mu||$ having finite total mass, one can decompose $\mu$ into positive and negative parts as $\mu=\mu_+-\mu_-$ and define
\begin{equation}
    Bal(\mu,K):=Bal(\mu_+,K)-Bal(\mu_-,K)\,.
\end{equation}

For fixed $X\in \mathbb{R}$ and any $x\ne X$, denote the Kelvin transform
\begin{equation}
    x^* = \frac{1}{x-X}+X\,.
\end{equation}
For a compact subset $K$ of $\mathbb{R}$ that does not contain $X$, denote
\begin{equation}
    K^* = \{x^*: x\in K\}\,.
\end{equation}

\begin{proposition}[{\cite[Lemma 2.4]{DOSW25}}]\label{prop_K}
    Let $X\in\mathbb{R}$, and $K\subset\mathbb{R}$ be a compact set that does not contain $X$. Let $\rho_{eq}$ be the equilibrium measure on $K^*$. Then
    \begin{equation}
        Bal(\delta_X,K)(x):=\frac{1}{C_0}\rho_{eq}(x^*)|x-X|^{s-2}\,,
    \end{equation}
    where $C_0$ is as appeared in Proposition \ref{prop_equi_1} for $K^*$.
\end{proposition}

Denote
\begin{equation}
    \omega_1:=\omega_{-1,1} = 2^{-s}\frac{\Gamma(1+s)}{\Gamma(\frac{1+s}{2})^2}\cdot (1-x^2)_+^{-\frac{1-s}{2}}\,,
\end{equation}
as the equilibrium measure on $[-1,1]$. Also denote
\begin{equation}\label{ballamdef}
    Bal_\lambda[\cdot] = Bal(\cdot,K_{\lambda,1})\,.
\end{equation}

The following result is crucial for the iterated balayage algorithm in \cite{DOSW25} and the one we will introduce in Section \ref{sec_iba}. As mentioned in \cite[Remark 6.3]{DOSW25}, currently it is only available in one dimension.
\begin{proposition}[{\cite[Lemma 6.4]{DOSW25}}]\label{prop_blam}
    Let $\lambda\in (0,1)$ and $\mu$ be a positive measure supported on $(-\lambda,\lambda)$. Then 
    \begin{equation}
        \frac{Bal_\lambda[\mu]}{\omega_1}
    \end{equation}
    is decreasing on $[\lambda,1]$.
\end{proposition}

\subsection{Some further results}

In this subsection we give some further results related to the equilibrium measure and the balayage operator.

\begin{lemma}\label{lem_eqpos}
    If $K$ is a finite union of disjoint closed intervals, and $I=[a,b]$ is one of the intervals, then the equilibrium measure $\rho$ on $K$ satisfies that
    \begin{equation}\label{lem_eqpos_1}
        \rho(x) = \omega_{a,b}(x)\Big(C-\int_{K\backslash I}   |x-y|^{-1} w(y)\rd{y}\Big)\,,
    \end{equation}
    on $I$, where $C>0$ and
    \begin{equation}
        w(y) = (b-a)^{s}\frac{\Gamma(\frac{1+s}{2})^2}{\Gamma(1+s)} \cdot\frac{\cos\frac{\pi s}{2}}{\pi}\big((b-y)(a-y)\big)^{\frac{1-s}{2}}\rho(y) \ge 0\,.
    \end{equation}
    As a consequence,
    \begin{equation}
        \frac{\rho(x)}{\omega_{a,b}(x)}
    \end{equation}
    is a strictly positive and real analytic function on $I$.
\end{lemma}

\begin{proof}
Write
\begin{equation}
    \rho_1 = \rho \chi_I,\quad \rho_2 = \rho \chi_{K\backslash I}\,.
\end{equation}
We have $W*\rho_1 + W*\rho_2 = C_0$ on $K$. Using \eqref{prop_baldef_2} and integrating against $\rho_2(y)\rd{y}$, for any $x\in I$ we have $(W*\rho_2)(x)=(W*\rho_3)(x)$, where
\begin{equation}\label{rho3}
    \rho_3(x) = \omega_{a,b}(x)\int_{K\backslash I}   |x-y|^{-1} w(y)\rd{y}\,,
\end{equation}
is supported on $I$. Then we see that $W*(\rho_1+\rho_3)=C_0$ on $I$. This implies that $\rho_1+\rho_3$ is a positive constant multiple of $\omega_{a,b}$, and \eqref{lem_eqpos_1} follows. Since $\frac{\rho_3}{\omega_{a,b}}$ is a real analytic function on $I$ by \eqref{rho3}, so is $\frac{\rho}{\omega_{a,b}}$.

Proposition \ref{prop_equi} already gives that $\lim_{x\rightarrow a^+}\frac{\rho(x)}{\omega_{a,b}(x)}=M>0$ and similarly for $\lim_{x\rightarrow b^-}$. Therefore, to show the positivity of $\frac{\rho(x)}{\omega_{a,b}(x)}$ on $[a,b]$, it suffices to do it for $x_0\in (a,b)$. Assume the contrary that $\rho(x_0)=0$, then $\rho(x)$ is $O((x-x_0)^2)$ near $x_0$ because it is real analytic. This allows us to calculate 
\begin{equation}
    0=(W*\rho)''(x_0) = \int_K s(s+1)|x_0-y|^{-s-2}\rho(y)\rd{y} > 0\,,
\end{equation}
leading to a contradiction.

\end{proof}

The following corollary is a direct consequence of Proposition \ref{prop_K} and Lemma \ref{lem_eqpos}.
\begin{corollary}\label{cor_eqpos}
    Let $K$ be a finite union of disjoint closed intervals, $X\notin K$, and $I=[a,b]$ is one of the intervals. Then 
    \begin{equation}
        \frac{Bal(\delta_X,K)}{\omega_{a,b}}
    \end{equation}
    is a strictly positive and real analytic function on $I$.
\end{corollary}

Then we give a simple comparison principle for the balayage operator.

\begin{lemma}\label{lem_bup}
    Let $K_1\subset K_2$ be compact subsets of $\mathbb{R}$, and $\mu$ be a positive measure with finite total mass. Then
    \begin{equation}\label{lem_bup_1}
        Bal(\mu,K_2) \le Bal(\mu,K_1)\,,
    \end{equation}
    on $K_1$.
\end{lemma}

\begin{proof}
Denote the positive measure
\begin{equation}
    \nu= Bal(\mu,K_2)\,,
\end{equation}
and write
\begin{equation}
    \nu = \nu\chi_{K_1} + \nu\chi_{K_2\backslash K_1} =: \nu_1+\nu_2\,.
\end{equation}
Then we have
\begin{equation}
    W*\nu_1 + W*\nu_2 = W*\mu,\quad \textnormal{q.e. }K_2\,.
\end{equation}
This implies
\begin{equation}
    W*\nu_1 = W*(\mu-\nu_2),\quad \textnormal{q.e. }K_1\,,
\end{equation}
i.e., 
\begin{equation}
    \nu_1 = Bal(\mu,K_1) - Bal(\nu_2,K_1)\,,
\end{equation}
since $\nu_1$ is supported on $K_1$. Since $\nu_2$ is positive, we see that $Bal(\nu_2,K_1)$ is also positive by Proposition \ref{prop_baldef}. This gives the conclusion.

\end{proof}

Finally we consider the minimizer $\rho$ of $E$ (defined in \eqref{E}) and analyze the continuity property of $V[\rho]$.
\begin{lemma}\label{lem_EU}
    Let $K$ be a finite union of disjoint closed intervals, and $U$ be a continuous function on $K$, and $m>0$. Then the energy $E$ (defined in \eqref{E}) has a unique minimizer $\rho$ in $m\cM(K)$. It is the unique element in $m\cM(K)$ satisfying the Euler-Lagrange condition
    \begin{equation}
        V[\rho]\le C_0,\quad\textnormal{on }\supp\rho;\qquad V[\rho]\ge C_0,\quad\textnormal{q.e. }K\,,
    \end{equation}
    for some $C_0$.
    
    Suppose $(a,b)$ is a maximal interval in $(\supp\rho)^c$, i.e., $(a,b)\subset (\supp\rho)^c$ and $a,b\in\supp\rho$. Then 
    \begin{equation}\label{lem_EU_1}
        \lim_{x\rightarrow a^+}V[\rho](x)=V[\rho](a),\quad \lim_{x\rightarrow b^-}V[\rho](x)=V[\rho](b).
    \end{equation}
    Similarly (3.26) holds for maximal intervals of the form $(-\infty,b)$ or $(a,\infty)$.
\end{lemma}

Here, \eqref{lem_EU_1} is analogous to \cite[Proposition 8.1 (iii)]{SW21}.

\begin{proof}
The existence, uniqueness and Euler-Lagrange condition can be found in \cite[Theorem 2.1]{DOSW}. 

To prove \eqref{lem_EU_1}, we only need to show $\lim_{x\rightarrow a^+}V[\rho](x)=V[\rho](a)$ by symmetry, and the results for $(-\infty,b)$ or $(a,\infty)$ can be proved similarly. First notice that $V[\rho]$ is continuous on $(a,b)$, and 
\begin{equation}
    \liminf_{x\rightarrow a^+}V[\rho](x)\ge V[\rho](a)\,,
\end{equation}
since $V[\rho]$ is lower-semicontinuous. Then we decompose $\rho$ into
\begin{equation}
    \rho = \rho \chi_{(-\infty,a]} + \rho \chi_{[b,\infty)} =: \rho_1 + \rho_2\,,
\end{equation}
and thus
\begin{equation}
    V[\rho] = W*\rho_1 + W*\rho_2 + U\,.
\end{equation}
Notice that $W*\rho_2$ and $U$ are continuous in a neighborhood of $a$, while $W*\rho_1$ is decreasing on $[a,\infty)$ since $W$ is decreasing on $(0,\infty)$ and $\rho_1$ is positive. Therefore we get
\begin{equation}
    \limsup_{x\rightarrow a^+}V[\rho](x)\le V[\rho](a)\,,
\end{equation}
and \eqref{lem_EU_1} is proved.
\end{proof}

\section{
{Proof of Lemma \ref{lem_lam1} via balayage}
}\label{sec_mulam}


{In this section we will prove Lemma \ref{lem_lam1} and study the properties of the signed measure $\mu_\lambda$.} We first notice that the uniqueness of $\mu_\lambda$ follows from the positive definite property of $W$, similar to the proof of Proposition \ref{prop_baldef}. 

To prove the existence of $\mu_\lambda$, we start from the explicit construction of $\mu_0$, then express the general $\mu_\lambda$ in terms of $\mu_0$ via the operator $Bal_\lambda$ (defined in \eqref{ballamdef}).

\subsection{
{Proof of Lemma \ref{lem_lam1} for $\lambda=0$}
}


{We prove Lemma \ref{lem_lam1} for $\lambda=0$ by giving the explicit formula for $\mu_0$ and verifying the characterizing property that $V[\mu_0]=0$ on $[-1,1]$.}

\begin{lemma}\label{lem_mu0}
    {Define the function $\mu_0$ by the explicit formula}
    \begin{equation}
        \mu_0(x) = C_{01}(1-x^2)_+^{-\frac{1-s}{2}} - C_{02}(1-x^2)_+^{\frac{1+s}{2}}\,,
    \end{equation}
    where 
    \begin{equation}
        C_{01}=\frac{3}{s(1-s)\Gamma(\frac{1-s}{2})\Gamma(\frac{1+s}{2})},\quad C_{02}=\frac{2}{1+s}C_{01}\,,
    \end{equation}
    are positive constants. {Then  $\mu_0$ satisfies that $V[\mu_0]=0$ on $[-1,1]$. Furthermore,} $\frac{\mu_0(x)}{\omega_1(x)}=C\mu_0(x)(1-x^2)_+^{\frac{1-s}{2}}$ is increasing on $(0,1)$. $\mu_0(x)$ is negative on $(0,\sqrt{\frac{1-s}{1+s}})$ and positive on $(\sqrt{\frac{1-s}{1+s}},1)$.
\end{lemma}

A more general explicit construction, for $U(\bx)=|\bx|^{2k},\,k\in \mathbb{Z}_{>0}$ in multi-dimensions, will be given in the upcoming work \cite{CHSS}.

\begin{proof}
We recall that the equilibrium measure of $W$ on $[-1,1]$ is a constant multiple of
\begin{equation}
    u_1(x) = (1-x^2)_+^{-\frac{1-s}{2}}\,.
\end{equation}
Thus we have the Euler-Lagrange condition
\begin{equation}\label{equ1}
    (W*u_1)(x) = C_1,\quad \textnormal{on }[-1,1]\,,
\end{equation}
where 
\begin{equation}\begin{split}
    C_1 = & (W*u_1)(1) 
    =  \int_{-1}^1 (1-y)^{-s}(1-y^2)^{-\frac{1-s}{2}}\rd{y} 
    =  \int_{-1}^1 (1-y)^{-\frac{1+s}{2}}(1+y)^{-\frac{1-s}{2}}\rd{y} \\
    = & \int_0^1 (1-z)^{-\frac{1+s}{2}}z^{-\frac{1-s}{2}}\rd{z} 
    =  B(\frac{1-s}{2},\frac{1+s}{2})\,.
\end{split}\end{equation}
Here in the fourth equality we used the change of variable $y=2z-1$, and $B(\cdot,\cdot)$ denotes the Beta function.

We also recall that the minimizer 
{of the energy $\frac{1}{2}\iint W(x-y)\rd{\rho(y)}\rd{\rho(x)} + \int x^2 \rd{\rho(x)}$}
is a rescaling of
\begin{equation}
    u_2 = (1-x^2)_+^{\frac{1+s}{2}}\,,
\end{equation}
(cf. \cite{CV11}). Thus we have the Euler-Lagrange condition
\begin{equation}\label{equ2}
    (W*u_2)(x) = C_2-C_3 x^2,\quad \textnormal{on }[-1,1]\,,
\end{equation}
for some constants $C_2$ and $C_3$. To find them, we use 
\begin{equation}\begin{split}
    C_2 = & (W*u_2)(0) 
    =  \int_{-1}^1 |y|^{-s}(1-y^2)^{\frac{1+s}{2}}\rd{y} 
    =  2\int_0^1 y^{-s}(1-y^2)^{\frac{1+s}{2}}\rd{y} \\
    = & \int_0^1 z^{-\frac{1+s}{2}}(1-z)^{\frac{1+s}{2}}\rd{y} 
    =  B(\frac{1-s}{2},\frac{3+s}{2})\,,
\end{split}\end{equation}
and
\begin{equation}\begin{split}
    C_2-C_3 = & (W*u_2)(1) 
    =  \int_{-1}^1 (1-y)^{-s}(1-y^2)^{\frac{1+s}{2}}\rd{y} 
    =  \int_{-1}^1 (1-y)^{\frac{1-s}{2}}(1+y)^{\frac{1+s}{2}}\rd{y} \\
    = & 4\int_0^1 (1-z)^{\frac{1-s}{2}}y^{\frac{1+s}{2}}\rd{z} 
    =  4B(\frac{3-s}{2},\frac{3+s}{2})\,,
\end{split}\end{equation}
i.e.,
\begin{equation}
    C_3 = B(\frac{1-s}{2},\frac{3+s}{2}) - 4B(\frac{3-s}{2},\frac{3+s}{2})\,.
\end{equation}

Recall that $0 = V[\mu_0](x) = (W*\mu_0)(x) + U(x) = (W*\mu_0)(x) - \frac{3}{1-s}x^2$ on $[-1,1]$. Doing a linear combination of \eqref{equ1} and \eqref{equ2}, we get 
\begin{equation}
    \mu_0 (x) = \frac{3C_2}{(1-s)C_1C_3}(1-x^2)_+^{-\frac{1-s}{2}}-\frac{3}{(1-s)C_3}(1-x^2)_+^{\frac{1+s}{2}} \,.
\end{equation}
Expressing the constants by gamma functions, we get the desired expression for $\mu_0$. The function $\mu_0(x)(1-x^2)_+^{\frac{1-s}{2}}=C_{01}(1-\frac{2}{1+s}(1-x^2))$ is clearly increasing on $(0,1)$. This also shows that $\mu_0(x)$ is negative on $(0,\sqrt{\frac{1-s}{1+s}})$ and positive on $(\sqrt{\frac{1-s}{1+s}},1)$.

\end{proof}

\subsection{
{Proof of Lemma \ref{lem_lam1} for $0<\lambda<1$ via balayage}
}


{To prove Lemma \ref{lem_lam1} for $0<\lambda<1$, we construct $\mu_\lambda$ from $\mu_0$ via the balayage operator, as done in the lemma below. We also show the relation between $\mu_\lambda$ with different values of $\lambda$. The continuity of $V[\mu_\lambda]$ stated in Lemma \ref{lem_lam1} is a direct consequence of Lemma \ref{lem_mucont} below, which controls the singularity of $\mu_\lambda$ near the boundary of $K_{\lambda,1}$.}

\begin{lemma}\label{lem_lam12}
    For any $\lambda\in (0,1)$, the formula 
    \begin{equation}\label{lem_lam12_2}
        \mu_\lambda = Bal_{\lambda}[\mu_0] = \mu_0\chi_{K_{\lambda,1}} + Bal_{\lambda}[\mu_0\chi_{(-\lambda,\lambda)}]\,,
    \end{equation}
    (where $\mu_0$ is given by Lemma \ref{lem_mu0}) satisfies \eqref{lem_lam1_1}, and thus $\mu_\lambda$ in Lemma \ref{lem_lam1} exists.
    
    Let $0\le \lambda_1 < \lambda_2 < 1$. Then
    \begin{equation}\label{lem_lam12_1}
        \mu_{\lambda_2} = Bal_{\lambda_2}[\mu_{\lambda_1}] = \mu_{\lambda_1}\chi_{K_{\lambda_2,1}} + Bal_{\lambda_2}[\mu_{\lambda_1}\chi_{(-\lambda_2,\lambda_2)}]\,.
    \end{equation}
\end{lemma}

\begin{proof}

Let $\mu_\lambda$ be given by \eqref{lem_lam12_2}, which is supported on $K_{\lambda,1}$. By the definition of $Bal_{\lambda}$, we have 
\begin{equation}
    V[\mu_\lambda] = V[\mu_0] = 0,\quad \textnormal{q.e. }K_{\lambda,1}\,.
\end{equation}
This equation actually holds everywhere on $K_{\lambda,1}$ due to the regularity result below, Lemma \ref{lem_mucont}, which implies that $W*\mu_\lambda$ is continuous on $\mathbb{R}$. Therefore $\mu_\lambda$ satisfies \eqref{lem_lam1_1}.

The relation \eqref{lem_lam12_1} can be derived by a similar argument.

\end{proof}

\begin{lemma}\label{lem_mucont}
    For every $\lambda\in(0,1)$, the $\mu_\lambda$ given by \eqref{lem_lam12_2} satisfies that  $\frac{\mu_\lambda}{\omega_{\lambda,1}}$ is $C^1$ on $[\lambda,1]$.
\end{lemma}

\begin{proof}

On the interval $(\lambda,1]$, since $\frac{\mu_0}{\omega_1}$ is a polynomial (by its explicit formula in Lemma \ref{lem_mu0}) and $\frac{\omega_1}{\omega_{\lambda,1}}$ is real analytic, we see that $\frac{\mu_0\chi_{K_{\lambda,1}}}{\omega_{\lambda,1}}$ is real analytic. Also, by Proposition \ref{prop_K} and Lemma \ref{lem_eqpos}, $\frac{Bal_{\lambda}[\mu_0\chi_{(-\lambda,\lambda)}]}{\omega_{\lambda,1}}$ is real analytic on $(\lambda,1]$. Therefore we only need to focus on the endpoint $\lambda$, i.e., show the $C^1$ property of $\frac{\mu_\lambda}{\omega_{\lambda,1}}$ in $[\lambda,\lambda+\epsilon]$ for small $\epsilon>0$.

We first design an auxiliary function
\begin{equation}
    u(x) = Bal_{\lambda/2}[\delta_0]\,.
\end{equation}
Corollary \ref{cor_eqpos} shows that $u$ is strictly positive in its support $K_{\lambda/2,1}$, and real analytic in the interior of $K_{\lambda/2,1}$.

We decompose
\begin{equation}
    u = u \chi_{(-\lambda,\lambda)} + u \chi_{K_{\lambda,1}} =: u_1 + u_2\,.
\end{equation}
Then we have
\begin{equation}
    Bal_\lambda[\delta_0] = Bal_\lambda[u] = Bal_\lambda[u_1] + u_2\,.
\end{equation}
Corollary \ref{cor_eqpos} shows that $\frac{Bal_\lambda[\delta_0]}{\omega_{\lambda,1}}$ is real analytic on $[\lambda,1]$, i.e., 
\begin{equation}
    \frac{Bal_\lambda[u_1](x) + u_2(x)}{\omega_{\lambda,1}}\,,
\end{equation}
is real analytic on $[\lambda,1]$. 

Then we start from \eqref{lem_lam12_2}, and rewrite it as
\begin{equation}
    \mu_\lambda = (\mu_0 - c_0 u)\chi_{K_{\lambda,1}} + Bal_{\lambda}[(\mu_0 - c_0 u)\chi_{(-\lambda,\lambda)}] + c_0 \big( u_2(x) + Bal_\lambda[u_1](x)\big)\,,
\end{equation}
where
\begin{equation}
    c_0 = \frac{\mu_0(\lambda)}{u(\lambda)}\,, 
\end{equation}
is chosen so that 
\begin{equation}
    v := \mu_0 - c_0 u\,,
\end{equation}
is smooth in a neighborhood of $\lambda$ and satisfies $v(\lambda)=0$. As a consequence, $\frac{v}{\omega_{\lambda,1}}$ is $C^1$ in a neighborhood of $\lambda$. Then it suffices to prove that $\frac{Bal_{\lambda}[v\chi_{(-\lambda,\lambda)}]}{\omega_{\lambda,1}}$ is $C^1$ on $[\lambda,\lambda+\epsilon]$.

Denote $\nu=Bal_{\lambda}[v\chi_{(-\lambda,\lambda)}]$. Since
\begin{equation}
    Bal(v\chi_{(-\lambda,\lambda)},[\lambda,1]) = Bal(\nu,[\lambda,1]) = Bal(\nu \chi_{[-1,-\lambda]},[\lambda,1]) + \nu \chi_{[\lambda,1]}\,,
\end{equation}
and $\frac{Bal(\nu \chi_{[-1,-\lambda]},[\lambda,1])}{\omega_{\lambda,1}}$ is clearly smooth on $[\lambda,1]$ by \eqref{prop_baldef_2}, we see that it suffices to prove the $C^1$ property of $\frac{Bal(v\chi_{(-\lambda,\lambda)},[\lambda,1])}{\omega_{\lambda,1}}$ on $[\lambda,\lambda+\epsilon]$. To do this, we use \eqref{prop_baldef_2} to write
\begin{equation}
    \frac{Bal(v\chi_{(-\lambda,\lambda)},[\lambda,1])}{\omega_{\lambda,1}} = C\int_{-\lambda}^\lambda \big((1-y)(\lambda-y)\big)^{\frac{1-s}{2}}  (x-y)^{-1} v(y)\rd{y}\,,
\end{equation}
for $x\in [\lambda,\lambda+\epsilon]$. We may differentiate with respect to $x$ on the RHS to get
\begin{equation}
    \frac{\rd}{\rd x}\Big(\frac{Bal(v\chi_{(-\lambda,\lambda)},[\lambda,1])}{\omega_{\lambda,1}}\Big) = C\int_{-\lambda}^\lambda \big((1-y)(\lambda-y)\big)^{\frac{1-s}{2}} \big(-(x-y)^{-2}v(y)\big) \rd{y}\,.
\end{equation}
Using that $v$ is smooth near $\lambda$ and $v(\lambda)=0$, the above calculation can be justified because the last integrand is bounded by a constant multiple of $(\lambda-y)^{\frac{1-s}{2}-2+1}$ near $\lambda$, making it locally integrable and uniformly bounded in $x$. It follows that the last integral is continuous on $[\lambda,\lambda+\epsilon]$, which gives the conclusion.

\end{proof}

\section{Continuity and limiting properties}\label{sec_mulam2}

In this section we give some continuity and limiting properties for the family $\mu_\lambda$ with respect to $\lambda$.

\begin{lemma}\label{lem_mulamcont}
    Fix $\lambda_0\in (0,1)$. For any $\alpha>0$, $\mu_\lambda$ converges to $\mu_{\lambda_0}$ as $\lambda\rightarrow \lambda_0^-$ uniformly on $[\lambda_0+\alpha,1-\alpha]$.
\end{lemma}

\begin{proof}
Take any $\lambda\in (0,\lambda_0)$ and denote $\epsilon=\lambda_0-\lambda$. Then \eqref{lem_lam12_1} gives
\begin{equation}\label{lam0cont1}
    \mu_{\lambda_0} = \mu_{\lambda}\chi_{K_{\lambda_0,1}} + Bal_{\lambda_0}[\mu_{\lambda}\chi_{(-\lambda_0,\lambda_0)}]\,,
\end{equation}
and \eqref{lem_lam12_2} gives
\begin{equation}\label{lam0cont2}
    \mu_{\lambda} = \mu_{0}\chi_{K_{\lambda,1}} + Bal_{\lambda}[\mu_{0}\chi_{(-\lambda,\lambda)}]\,.
\end{equation}
First, \eqref{lam0cont2} with Lemma \ref{lem_bup} and \eqref{prop_baldef_2} shows that
\begin{equation}\label{alpest1}\begin{split}
    \int |\mu_{\lambda}|\chi_{(-\lambda_0,\lambda_0)}\rd{x} \le & 2\int_{\lambda}^{\lambda_0} |\mu_0(x)|\rd{x}  + 2\int_{\lambda}^{\lambda_0}Bal_{\lambda}[|\mu_{0}|\chi_{(-\lambda,\lambda)}](x)\rd{x} \\
    \lesssim & \epsilon  + \int_{\lambda}^{\lambda_0}Bal(\chi_{(-\lambda,\lambda)},[\lambda,1])(x)\rd{x} \\
    \lesssim & \epsilon  + \int_{\lambda}^{\lambda_0}\int_{-\lambda}^\lambda \big((1-y)(\lambda-y)\big)^{\frac{1-s}{2}}  (x-y)^{-1}\rd{y}\cdot \big((1-x)(x-\lambda)\big)_+^{-\frac{1-s}{2}}\rd{x} \\
    \lesssim & \epsilon  + \int_{\lambda}^{\lambda_0}\int_{-\lambda}^\lambda (\lambda-y)^{\frac{1-s}{2}}  (x-y)^{-1}\rd{y}\cdot (x-\lambda)^{-\frac{1-s}{2}}\rd{x} \\
    \lesssim & \epsilon  + \int_{\lambda}^{\lambda_0}\int_{-\lambda}^\lambda (\lambda-y)^{-\frac{1+s}{2}}  \rd{y}\cdot (x-\lambda)^{-\frac{1-s}{2}}\rd{x} \\
    \lesssim & \epsilon  + \int_{\lambda}^{\lambda_0}(x-\lambda)^{-\frac{1-s}{2}}\rd{x} \\
    \lesssim & \epsilon^{\frac{1+s}{2}}\,.
\end{split}\end{equation}
Then, from \eqref{lam0cont1}, for $x\in [\lambda_0+\alpha,1-\alpha]$, we have
\begin{equation}\begin{split}
    |\mu_{\lambda_0}(x)-\mu_\lambda(x)| = & |Bal_{\lambda_0}[\mu_{\lambda}\chi_{(-\lambda_0,\lambda_0)}](x)| \\
    \lesssim  & Bal(|\mu_{\lambda}|\chi_{(-\lambda_0,\lambda_0)},[\lambda,1])(x) \\
    \lesssim  & \int_{-\lambda_0}^{\lambda_0} \big((1-y)(\lambda_0-y)\big)^{\frac{1-s}{2}}  (x-y)^{-1}|\mu_{\lambda}(y)|\rd{y} \cdot \big((1-x)(x-\lambda_0)\big)_+^{-\frac{1-s}{2}} \\
    \lesssim  & \int_{-\lambda_0}^{\lambda_0} |\mu_{\lambda}(y)|\rd{y} \\ 
    \lesssim & \epsilon^{\frac{1+s}{2}}\,,
\end{split}\end{equation}
where we allow the implied constant to depend on the given number $\alpha$ (thus $x-y,1-x,x-\lambda_0$ are bounded from below). This finishes the proof.

\end{proof}

\begin{lemma}\label{lem_Fcont}
    The function $F$ defined in \eqref{F} is continuous on $(0,1)$.
\end{lemma}

\begin{proof}
For fixed $0<\lambda<\lambda+\epsilon<1$, we may proceed similarly as in the proof of Lemma \ref{lem_mulamcont} to see that
\begin{equation}
    \int |\mu_\lambda|\chi_{(-(\lambda+\epsilon),\lambda+\epsilon)}\rd{x} \lesssim \epsilon^{\frac{1+s}{2}}\,.
\end{equation}
Combining with \eqref{lem_lam12_1} and the fact that balayage operators decrease total mass (cf. Proposition \ref{prop_equi}), we see that
\begin{equation}
    \int |\mu_{\lambda+\epsilon}-\mu_\lambda|\rd{x} \lesssim \epsilon^{\frac{1+s}{2}}\,.
\end{equation}
Therefore
\begin{equation}
    |F(\lambda+\epsilon)-F(\lambda)| = \left|\int_{-(\lambda+\epsilon)}^{\lambda+\epsilon} W*(\mu_{\lambda+\epsilon}-\mu_\lambda)\rd{x}\right| \le \|W\|_{L^1([-2,2])} \cdot \int |\mu_{\lambda+\epsilon}-\mu_\lambda|\rd{x} \lesssim \epsilon^{\frac{1+s}{2}}\,,
\end{equation}
which finishes the proof.

\end{proof}

\begin{lemma}\label{lem_lampos}
    There exists $\lambda_+\in [\sqrt{\frac{1-s}{1+s}},1)$ such that $\mu_\lambda$ is positive for any $\lambda \in (\lambda_+,1)$, and $\mu_\lambda$ is not everywhere positive for any $\lambda\in (0,\lambda_+)$.
\end{lemma}

\begin{proof}
The explicit formula for $\mu_0$ in Lemma \ref{lem_mu0}, together with \eqref{lem_lam12_2} and the positivity of the operator $Bal_\lambda$, shows that $\mu_\lambda(x)$ is not everywhere positive for any $\lambda \in (0,\sqrt{\frac{1-s}{1+s}})$. Therefore, it suffices to prove that $\mu_\lambda$ is positive for $\lambda$ sufficiently close to 1, since the positivity of some $\mu_\lambda$ would imply that for $\mu_{\lambda'}$ for any $\lambda'>\lambda$. 

Denote $\epsilon=1-\lambda>0$, and assume $\epsilon\le \frac{1}{2}$. Let $\rho_{\lambda,m}$ be the unique minimizer of the energy $E$ with total mass $m>0$ on $K=K_{\lambda,1}$, given by Lemma \ref{lem_EU}. By symmetry, we have
\begin{equation}    
    \int_{[\lambda,1]}\rd{\rho_{\lambda,m}} = \frac{m}{2}\,.
\end{equation}
We claim that $\supp\rho_{\lambda,m}=K_{\lambda,1}$ if
\begin{equation}
    m \ge m_0:= \max\{C_{s1} \epsilon^{s+1},C_{s2} \epsilon^{s+2}\}\,,
\end{equation}
where $C_{s1},C_{s2}$ are some constants independent of $\epsilon$, to be determined. To see this, assume the contrary that $\supp\rho_{\lambda,m}\ne K_{\lambda,1}$, then there are three possibilities:
\begin{itemize}
    \item There exists a maximal interval $(x_1,x_2)$ in $(\lambda,1)\backslash \supp\rho_{\lambda,m}$, i.e., $\lambda<x_1<x_2<1$, $x_1,x_2\in \supp\rho_{\lambda,m}$ and $(x_1,x_2)\cap\supp\rho_{\lambda,m}=\emptyset$. For $x\in (x_1,x_2)$, we estimate
    \begin{equation}
        V[\rho_{\lambda,m}]''(x) = (W''*\rho_{\lambda,m})(x) + U''(x) \ge \frac{m}{2} s(s+1)\epsilon^{-s-2} - \frac{6}{1-s} > 0\,,
    \end{equation}
    (by taking $C_{s2}=\frac{24}{(1-s)s(s+1)}$) which leads to a contradiction against the Euler-Lagrange condition by maximum principle, in view of the continuity in \eqref{lem_EU_1}.
    \item There exists a maximal interval $(\lambda,x_1)$ in $(\lambda,1)\backslash \supp\rho_{\lambda,m}$. For $x\in (\lambda,x_1)$, using $\epsilon\le \frac{1}{2}$, we estimate
    \begin{equation}
        V[\rho_{\lambda,m}]'(x) = (W'*\rho_{\lambda,m})(x) + U'(x) \ge \frac{m}{2} s(\epsilon^{-s-1}-1) - \frac{6}{1-s} \ge \frac{m}{4} s\epsilon^{-s-1} - \frac{6}{1-s} > 0\,,
    \end{equation}
    (by taking $C_{s1}=\frac{48}{(1-s)s}$) which leads to a contradiction against the Euler-Lagrange condition.
    \item There exists a maximal interval $(x_1,1)$ in $(\lambda,1)\backslash \supp\rho_{\lambda,m}$. One can directly get $V[\rho_{\lambda,m}]'(x)<0$ on $(x_1,1)$ which leads to a contradiction.
\end{itemize}
This proves the claim. We also notice that $m_0=C_{s1} \epsilon^{s+1}$ for sufficiently small $\epsilon$.

Then we see that the constant value $C_0=V[\rho_{\lambda,m_0}]$ on $K_{\lambda,1}$ is bounded by 
\begin{equation}
    C_0 = \frac{1}{2\epsilon}\int_{K_{\lambda,1}}\big((W*\rho_{\lambda,m_0})(x) +U(x)\big)\rd{x} \le \frac{1}{\epsilon}m_0\cdot \int_{-\epsilon/2}^{\epsilon/2} |y|^{-s}\rd{y} - \frac{3}{1-s}\cdot \frac{1}{4} = 2^s C_{s1}\epsilon -\frac{3}{4(1-s)} < 0\,,
\end{equation}
when $\epsilon$ is sufficiently small. Then, adding $\rho_{\lambda,m_0}$ with a positive constant multiple of the equilibrium measure on $K_{\lambda,1}$, we get $\mu_\lambda$ which is positive.

\end{proof}

{Finally we calculate the limit of $F(\lambda)$ as $\lambda$ approaches 1, which will be crucially used in the proof of Theorem \ref{thm_main} in Section \ref{sec_main}.}

\begin{lemma}\label{lem_Fneg}
    The function $F$ satisfies
    \begin{equation}
        \lim_{\lambda\rightarrow 1^-} F(\lambda) = \int_{-1}^1 U(x)\rd{x} < 0\,.
    \end{equation}
\end{lemma}

\begin{proof}
By Lemma \ref{lem_lampos}, $\mu_\lambda$ is positive for $\lambda$ sufficiently close to 1. Then, denoting $\epsilon=1-\lambda$, we have
\begin{equation}\begin{split}
    0 = & \frac{1}{\epsilon}\int_{[\lambda,1]} ((W*\mu_\lambda)(x) + U(x))\rd{x} \\ \ge & \frac{1}{\epsilon}\int_{0}^{\epsilon} y^{-s}\rd{y} \cdot \int_{[\lambda,1]} \rd{\mu_\lambda} - \frac{3}{1-s} \\ = & \frac{1}{2(1-s)}\epsilon^{-s} \int_{K_{\lambda,1}} \rd{\mu_\lambda} - \frac{3}{1-s}\,,
\end{split}\end{equation}
which implies
\begin{equation}
    \int_{K_{\lambda,1}} \rd{\mu_\lambda} \le 6\epsilon^s\,.
\end{equation}
Then we get
\begin{equation}
    F(\lambda)-\int_{-1}^1 U(x)\rd{x} = \int_{-1}^1 (W*\mu_\lambda)(x)\rd{x} \le \|W\|_{L^1([-2,2])} \cdot \int_{K_{\lambda,1}} \rd{\mu_\lambda} \lesssim \epsilon^s\,,
\end{equation}
which gives the conclusion.
\end{proof}

\section{The iterated balayage algorithm (IBA)}\label{sec_iba}

In this section we give the IBA as a sequence $\{\mu_{\lambda_j}\}$, where each $\mu_{\lambda_j}$ is the balayage of $\mu_{\lambda_{j-1}}$ onto the support of the positive part of $\mu_{\lambda_{j-1}}$. To be precise, we have the following.

\begin{proposition}\label{prop_iba}
    Define $\lambda_0=0$ and $\lambda_j,\,j=1,2,\dots$ iteratively by
    \begin{equation}
        \lambda_j = \inf\{x>\lambda_{j-1}: \mu_{\lambda_{j-1}}(x) \ge 0 \}\,.
    \end{equation}
    Then we have
    \begin{enumerate}
        \item For each $j\ge 0$, $\lambda_j$ is well-defined with $\lambda_j\in [0,1)$. $\mu_{\lambda_j}$ is sign-changing, and $\frac{\mu_{\lambda_j}}{\omega_1}$ is strictly increasing on $x\in(\lambda_j,1)$.
        \item $\lambda_{j+1}>\lambda_j$, and $\lambda_\infty:=\lim_{j\rightarrow\infty}\lambda_j \in (0,1)$. $\frac{\mu_{\lambda_\infty}}{\omega_1}$ is positive and strictly increasing on $(\lambda_\infty,1)$. $\lambda_\infty$ is the same as $\lambda_+$ given by Lemma \ref{lem_lampos}.
        \item  $W*\mu_{\lambda_\infty}$ is $C^1$ on $(-1,1)$.
        \item $V[\mu_{\lambda_\infty}] > 0$ on $(-\lambda_\infty,\lambda_\infty)$. As a consequence, $F(\lambda_\infty)>0$.
    \end{enumerate}    
\end{proposition}

\begin{proof}

{\bf Item 1}:

We prove Item 1 by induction on $j$. The case $j=0$ is clearly true due to Lemma \ref{lem_mu0}. Suppose this statement is true for some $j$. Since $\mu_{\lambda_j}$ is sign-changing and $\frac{\mu_{\lambda_j}}{\omega_1}$ is strictly increasing on $x\in(\lambda_j,1)$, we see that  $\lambda_{j+1}\in [0,1)$ is well-defined, satisfying $\lambda_{j+1}>\lambda_j$.

To analyze the monotonicity of $\frac{\mu_{\lambda_{j+1}}}{\omega_1}$ on $(\lambda_{j+1},1)$, we use \eqref{lem_lam12_1} to write
\begin{equation}\label{mulamj}
    \mu_{\lambda_{j+1}} = \mu_{\lambda_j}\chi_{K_{\lambda_{j+1},1}} + Bal_{\lambda_{j+1}}[\mu_{\lambda_j}\chi_{(-\lambda_{j+1},\lambda_{j+1})}] =: I_1 + I_2\,.
\end{equation}
By the induction hypothesis, $\frac{\mu_{\lambda_j}}{\omega_1}$ is strictly increasing on $(\lambda_j,1)$, with $\mu_{\lambda_j}(\lambda_{j+1})=0$ due to its continuity from Lemma \ref{lem_mucont}, the definition of $\lambda_{j+1}$ and the fact that $\lambda_{j+1}>\lambda_j$. Therefore $\mu_{\lambda_j}< 0$ on $(\lambda_j,\lambda_{j+1})$ and $\mu_{\lambda_j}> 0$ on $(\lambda_{j+1},1)$. 

On the RHS of \eqref{mulamj}, $\frac{I_1}{\omega_1}$ is strictly increasing on $(\lambda_{j+1},1)$. $\frac{I_2}{\omega_1}$ is also increasing on $(\lambda_{j+1},1)$ due to $\mu_{\lambda_j}\chi_{(-\lambda_{j+1},\lambda_{j+1})}< 0$ and Proposition \ref{prop_blam}. Therefore we conclude that $\frac{\mu_{\lambda_{j+1}}}{\omega_1}$ is strictly increasing on $(\lambda_{j+1},1)$. 

Notice that $\mu_{\lambda_{j+1}}$ cannot be a negative measure, because otherwise one would have that $\mu_{\lambda}$ is negative for any $\lambda\in (\lambda_{j+1},1)$ (due to \eqref{lem_lam12_1} and the positivity of the operator $Bal_\lambda$), contradicting Lemma \ref{lem_lampos}. Then we show that $\mu_{\lambda_{j+1}}$ cannot be a positive measure. In fact, we already proved that $\mu_{\lambda_j} < 0$ on $(\lambda_j,\lambda_{j+1})$ and $\mu_{\lambda_j} > 0$ on $(\lambda_{j+1},1)$. By Lemma \ref{lem_mucont}, we see that $I_1(x)$ (defined in \eqref{mulamj}) is continuous on $[\lambda_{j+1},\lambda_{j+1}+\epsilon]$ for small $\epsilon>0$, with $I_1(\lambda_{j+1})=0$. On the other hand, by Corollary \ref{cor_eqpos} with a negative linear combination with weight $\mu_{\lambda_j}(X)\chi_{(-\lambda_{j+1},\lambda_{j+1})}(X)$, we see that $\lim_{x\rightarrow \lambda_{j+1}^-}I_2(x)=-\infty$. Therefore $\mu_{\lambda_{j+1}}(x)<0$  when $x-\lambda_{j+1}>0$ is sufficiently small. Therefore we see that $\mu_{\lambda_{j+1}}$ is sign-changing, and the induction is finished.

{\bf Item 2}:

The property $\lambda_{j+1}>\lambda_j$ was already proved in the middle of the previous induction. Then it follows that $\lambda_\infty:=\lim_{j\rightarrow\infty}\lambda_j$ exists and is in $(0,1]$. Since each $\mu_{\lambda_j}$ is sign-changing, we see that $\lambda_j\le \lambda_+<1$ by Lemma \ref{lem_lampos}, and thus we have $\lambda_\infty\le \lambda_+<1$. 

We notice that $\mu_{\lambda_j}\ge 0$ on $(\lambda_\infty,1)$ because otherwise one would have $\lambda_{j+1} > \lambda_\infty$ which is a contradiction. Then the nonnegative and increasing property of $\frac{\mu_{\lambda_\infty}}{\omega_1}$ on $(\lambda_\infty,1)$ follows from the same property of $\mu_{\lambda_j}$ and taking limit $j\rightarrow\infty$ by using Lemma \ref{lem_mulamcont}. This nonnegative property also implies $\lambda_\infty\ge\lambda_+$ by the definition of $\lambda_+$ in Lemma \ref{lem_lampos}. Therefore $\lambda_\infty=\lambda_+$. 

The increasing property of $\frac{\mu_{\lambda_\infty}}{\omega_1}$ is strict, because the proof of Item 1 shows that $\frac{\mu_{\lambda_{j+1}}-\mu_{\lambda_j}}{\omega_1}$ is increasing on $(\lambda_\infty,1)$ for every $j$. As a consequence, $\frac{\mu_{\lambda_\infty}}{\omega_1}$ is strictly positive on $(\lambda_\infty,1)$. This finishes the proof of Item 2.

{\bf Item 3}:


Lemma \ref{lem_mucont} gives that 
\begin{equation}
    u(x) := \mu_{\lambda_\infty}(x)\cdot \big((x-\lambda_\infty)(1-x)\big)^{\frac{1-s}{2}}\,,
\end{equation}
is $C^1$ on $[\lambda_\infty,1]$, and it is also nonnegative. Therefore the positive and strictly increasing property of $\frac{\mu_{\lambda_\infty}(x)}{\omega_1(x)} = C u(x)\cdot \big(\frac{1+x}{x-\lambda_\infty}\big)^{\frac{1-s}{2}}$ on $(\lambda_\infty,1)$ enforces that $u(\lambda_\infty)=0$. As a consequence, $\mu_{\lambda_\infty}$ is $C^1$ on the open interval $(\lambda_\infty,1)$, and
\begin{equation}
    |\mu_{\lambda_\infty}'(x)| \lesssim (x-\lambda_\infty)^{-\frac{1-s}{2}}\,,
\end{equation}
for $x-\lambda_\infty>0$ small. Then it follows that $W*\mu_{\lambda_\infty}'$ is continuous near $\lambda_\infty$ by calculating $-s-\frac{1-s}{2} =-\frac{1+s}{2} > -1$. The continuity of $W*\mu_{\lambda_\infty}'$ elsewhere is clear.

{\bf Item 4}:


From Item 3, we see that $v(x):=V[\mu_{\lambda_\infty}](x)$ is $C^1$ on $(-1,1)$. Therefore its definition, \eqref{lem_lam1_1}, implies that
\begin{equation}\label{vpm}
    v(\pm\lambda_\infty) = v'(\pm\lambda_\infty) = 0\,.
\end{equation}
Since $W$ is smooth away from 0, $U$ is smooth, and $\supp\mu_{\lambda_\infty}=K_{\lambda_\infty,1}$, we see that $v$ is smooth on $(-\lambda_\infty,\lambda_\infty)$. Notice that
\begin{equation}
    W^{(4)}(x) = (-s)(-1-s)(-2-s)(-3-s)|x|^{-4-s}>0,\,\forall x\ne 0; \quad U^{(4)}=0\,,
\end{equation}
we see that $v^{(4)}>0$ on $(-\lambda_\infty,\lambda_\infty)$. Therefore $v''(x)$ is strictly convex on $(-\lambda_\infty,\lambda_\infty)$. Since 
\begin{equation}
    \int_{-\lambda_\infty}^{\lambda_\infty} v''(x)\rd{x} = v'(\lambda_\infty)-v'(-\lambda_\infty)=0\,,
\end{equation}
we see that $v''$ has to be sign-changing. Since $v''(x)$ is also an even function, there exists $y\in (0,\lambda_\infty)$ such that $v''>0$ on $(-\lambda_\infty,-y)\cup (y,\lambda_\infty)$, and $v''<0$ on $(-y,y)$ (see Figure \ref{fig2} for illustration). Then, by doing integration and using \eqref{vpm}, we see that $v(x)>0$ on $(-\lambda_\infty,-y]\cup [y,\lambda_\infty)$. Then it follows that $v(x)>0$ on $(-y,y)$ by maximum principle. The fact that $F(\lambda_\infty)>0$ then follows from the definition of $F$ in \eqref{F}. 

\end{proof}

\begin{remark}[Comparing Proposition \ref{prop_iba} with {\cite[Section 6]{DOSW25}}]\label{rmk_compare}
    We compare our IBA in Proposition \ref{prop_iba} with the one in \cite[Section 6]{DOSW25} as follows. Their similarity is that both IBAs utilize the balayage operator $Bal_\lambda$, and rely on the monotonicity as in Item 1 of Proposition \ref{prop_iba}, which essentially comes from Proposition \ref{prop_blam}. They both stop at the smallest $\lambda$ value such that $\mu_\lambda$ is positive.

    The first difference between these two IBAs is that in Proposition \ref{prop_iba} we utilize the relation \eqref{lem_lam12_1} to build the iteration, while in \cite{DOSW25} they need to subtract a multiple of $Bal_\lambda[\omega_1]$ in each iteration. This is because \cite{DOSW25} seeks for certain minimizer with total mass one for the interaction potential $W$, and thus they need to keep the total mass unchanged through the iteration. They are allowed to find positive $\mu_\lambda$ with $V[\mu_\lambda]$ equal to any constant $C$ on its support. This makes their IBA more similar to the classical version using the $\widetilde{Bal}$ operator in \eqref{eq31}. However, we insist on finding some positive $\mu_\lambda$ such that $V[\mu_\lambda]=0$, cf. Lemma \ref{lem_lam2}, but do not care about conserving the total mass.

    The second difference is that the analogue of $\mu_{\lambda_\infty}$ already gives the minimizer that \cite{DOSW25} seeks for, but we need to proceed further to larger $\lambda$ to find the $\mu_\lambda$ with $F(\lambda)=0$. This is clear from Item 4 of Proposition \ref{prop_iba}, which actually implies that $\mu_\lambda$ is the minimizer for $E$ (defined in \eqref{E}) among all positive measures on $[-1,1]$ with the same total mass, but still has $F(\lambda_\infty)>0$. Therefore we need the continuity result Lemma \ref{lem_Fcont} and the limit result Lemma \ref{lem_Fneg} to find the $\lambda$ we want, as will be done in the next section.
\end{remark}

\section{Proof of Theorem \ref{thm_main}}\label{sec_main}

Proposition \ref{prop_iba} gives $F(\lambda_\infty)>0$ where $\lambda_\infty\in (0,1)$. Lemma \ref{lem_Fneg} gives that $\lim_{\lambda\rightarrow 1^-} F(\lambda)<0$. Since $F$ is continuous by Lemma \ref{lem_Fcont}, we see that there exists $\lambda_*\in (\lambda_\infty,1)$ such that $F(\lambda_*)=0$. 

Since $\mu_{\lambda_\infty}$ is positive by Proposition \ref{prop_iba}, we see that $\mu_{\lambda_*}$ is also positive by Lemma \ref{lem_lampos}. Therefore we may apply Lemma \ref{lem_lam2} to see that 
\begin{equation}\label{rhorescale}
    \rho(x):=R_2^{1+s}\mu_{\lambda_*}(\frac{x}{R_2}),\quad R_2=\Big(\int_{K_{\lambda_*,1}}\mu_{\lambda_*}(x)\rd{x}\Big)^{-1/(2+s)} \,,
\end{equation}
is a probability measure on $K_{R_1,R_2}$ satisfying \eqref{EL} (where $R_1=\lambda_* R_2$). 

In the rest of the proof we verify \eqref{EL2} for $\rho$, which would imply that $\rho$ is the unique minimizer of $\cE$ due to Proposition \ref{prop_EL}. Denote
\begin{equation}
    v = V[\mu_{\lambda_*}]\,,
\end{equation}
which satisfies $v=0$ on $K_{\lambda_*,1}$, and the regularity of $\mu_{\lambda_*}$ from Lemma \ref{lem_mucont} shows that $v$ is continuous on $(-1,1)$. Notice that by \eqref{lem_lam12_1}, we have
\begin{equation}
    \mu_{\lambda_*} = \mu_{\lambda_\infty}\chi_{K_{\lambda_*,1}} + Bal_{\lambda_*}[\mu_{\lambda_\infty}\chi_{(-\lambda_*,\lambda_*)}]\,,
\end{equation}
where $\lambda_*>\lambda_\infty$ and $\mu_{\lambda_\infty}$ is strictly positive in $(\lambda_\infty,1)$ due to Item 2 of Proposition \ref{prop_iba}. Thus, combining with Lemma \ref{lem_mucont}, we necessarily have that
\begin{equation}\label{ulamst}
    \frac{\mu_{\lambda_*}}{\omega_{\lambda_*,1}}
\end{equation}
is $C^1$ on $[\lambda_*,1]$ and strictly positive. 


Then we have
\begin{equation}
    \lim_{x\rightarrow \lambda_*^-} v'(x) = \infty,\quad \lim_{x\rightarrow \lambda_*^-} v''(x) = \infty\,,
\end{equation}
by utilizing the singularity of $\mu_{\lambda_*}$ near $\lambda_*$ that behaves like $(x-\lambda_*)^{-\frac{1-s}{2}}$. We also have $v^{(4)}>0$ on $(-\lambda_*,\lambda_*)$ similar to the proof of Item 4 of Proposition \ref{prop_iba}. Therefore $v''(x)$ is even and strictly convex on $(-\lambda_*,\lambda_*)$. $v''$ cannot be everywhere nonnegative on $(-\lambda_*,\lambda_*)$ because otherwise one would have $v<0$ on $(-\lambda_*,\lambda_*)$, leading to $F(\lambda_*)<0$, a contradiction. Therefore there exists $y\in (0,\lambda_*)$ such that $v''>0$ on $(-\lambda_*,-y)\cup (y,\lambda_*)$, and $v''<0$ on $(-y,y)$ (see Figure \ref{fig2} for illustration). Then we deduce that there exists $z\in (0,\lambda_*)$ such that $v<0$ on $(-\lambda_*,-z)\cup (z,\lambda_*)$, and $v>0$ on $(-z,z)$.

We also notice that for any $x>1$,
\begin{equation}
    v'(x) = \int_{K_{\lambda_*,1}}W'(x-t)\mu_{\lambda_*}(t)\rd{t} + U'(x) < 0\,,
\end{equation}
because $W'(x-t)<0$ when $x-t>0$, and $U'(x)<0$. It follows that $v(x)<0$ on $(-\infty,1)\cup (1,\infty)$. 

Then we denote
\begin{equation}
    \tilde{v} = \cW*\rho+\cU \,,
\end{equation}
which satisfies \eqref{EL}, i.e., $\tilde{v}=C_0$ on $K_{R_1,R_2}$. Notice that
\begin{equation}
    -\frac{1}{1-s}\tilde{v}''(x) = V[\rho](x) = R_2^2 v(\frac{x}{R_2})\,,
\end{equation}
due to the calculation in \eqref{VR2calc}. Then it follows from the sign properties of $v$ that
\begin{equation}
    \tilde{v}''>0,\,\textnormal{on }(-\infty,R_2)\cup (R_2,\infty)\cup(-R_1,-z R_2)\cup (z R_2,R_1);\quad \tilde{v}''<0,\,\textnormal{on } (-z R_2,z R_2)\,.
\end{equation}
Then one can proceed as in the proof of Item 4 of Proposition \ref{prop_iba} (with $\tilde{v},z R_2,R_1$ in the place of $v,y,\lambda_\infty$ respectively),  to see that $\tilde{v}\ge C_0$ on $\mathbb{R}$. This gives \eqref{EL2}.

The regularity and positivity of $\rho$ stated in Theorem \ref{thm_main} is a consequence of that of \eqref{ulamst} (which we have proved), and the fact that $\rho$ is a rescaling of $\mu_{\lambda_*}$ by \eqref{rhorescale}.

\bibliographystyle{alpha}
\bibliography{minimizer_book_bib.bib}

@article{Lub93,
	author = {D. S. Lubinsky},
	journal = {Acta Appl. Math.},
	pages = {121-164},
	title = {An Update on Orthogonal Polynomials and Weighted Approximation on the real line},
	volume = {33},
	year = {1993}}

@book{Mha,
	author = {H. N. Mhaskar},
	publisher = {World Scientific, Singapore},
	title = {Introduction to the theory of weighted polynomial approximation},
	year = {1996}}

@article{Lub,
	author = {D. S. Lubinsky},
	journal = {Surveys in Approximation Theory},
	pages = {1-105},
	title = {A survey of weighted polynomial approximation with exponential weights},
	volume = {3},
	year = {2007}}

@book{LL,
	author = {E. Levin and D. S. Lubinsky},
	publisher = {Canadian Mathematical Society},
	series = {CMS Books in Mathematics},
	title = {Orthogonal Polynomials for Exponential weights},
	year = {2001}}

@book{Deift,
	author = {P. Deift},
	publisher = {American Mathematical Society},
	series = {Courant Lecture notes},
	title = {Orthogonal Polynomials and Random Matrices: A {R}iemann-{H}ilbert Approach},
	year = {2000}}

@article{CHSS,
	author = {J. A. Carrillo and Y. Huang and E. B. Saff and R. Shu},
	journal = {in preparation},
	title = {On the Support of Minimizers of Interaction Energies with Confinement}}

@article{DKM98,
	author = {P. Deift and T. Kriecherbauer and K. T-R McLaughlin},
	journal = {J. Approx. Theory},
	pages = {388--475},
	title = {New results on the equilibrium measure for logarithmic potentials in the presence of an external field},
	volume = {95},
	year = {1998}}

@book{Gak,
	author = {F. D. Gakhov},
	publisher = {Pergamon Press, Oxford},
	title = {Boundary Value Problems},
	year = {1966}}

@article{BD12,
	author = {D. Benko and P. D. Dragnev},
	journal = {Constructive Approximation},
	pages = {191-214},
	title = {Balayage Ping-Pong: A Convexity of Equilibrium Measures},
	volume = {36},
	year = {2012}}

@article{BDD06,
	author = {D. Benko and S. B. Damelin and P. D. Dragnev},
	journal = {Electronic Transactions on Numerical Analysis},
	pages = {27-40},
	title = {On the support of the equilibrium measure for arcs of the unit circle and real intervals},
	volume = {25},
	year = {2006}}

@article{DDK01,
	author = {S. B. Damelin and P. D. Dragnev and A. B. J. Kuijlaars},
	journal = {Pacific Journal of Mathematics},
	number = {2},
	pages = {303-321},
	title = {The support of the equilibrium measure for a class of external fields on a finite interval},
	volume = {199},
	year = {2001}}

@article{DK99,
	author = {S. B. Damelin and A. B. J. Kuijlaars},
	journal = {Trans. Amer. Math. Soc.},
	pages = {4561-4584},
	title = {The support of the extremal measure for monomial external fields on $[-1,1]$},
	volume = {351},
	year = {1999}}

@article{KD99,
	author = {A. B. J. Kuijlaars and P. D. Dragnev},
	journal = {Proc. Amer. Math. Soc.},
	pages = {1065--1074},
	title = {Equilibrium problems associated with fast decreasing polynomials},
	volume = {127},
	year = {1999}}

@book{StahlTotik,
	author = {H. Stahl and V. Totik},
	publisher = {Cambridge University Press, Cambridge},
	title = {General Orthogonal Polynomials},
	year = {1992}}

@book{Landkof,
	author = {N. S. Landkof},
	publisher = {Springer},
	title = {Foundations of modern potential theory},
	volume = {180},
	year = {1972}}

@article{Wallin,
	author = {H. Wallin},
	journal = {Annales de l'institut Fourier},
	number = {2},
	pages = {71-90},
	title = {Regularity properties of the equilibrium distribution},
	volume = {15},
	year = {1965}}

@article{DOSW25,
	author = {P. D. Dragnev and R. Orive and E. B. Saff and F. Wielonsky},
	journal = {preprint, arXiv:2501.01208},
	title = {Riesz equilibrium on a ball in the external field of a point charge}}

@article{CMSVW,
	author = {D. Chafa\"i and R. W. Matzke and E. B. Saff and M. Q. H. Vu and R. S. Womersley},
	journal = {Potential Analysis},
	title = {Riesz Energy with a Radial External Field: When is the Equilibrium Support a Sphere},
	year = {2024}}

@article{Shu_convex,
	author = {R. Shu},
	journal = {preprint, arXiv:2503.09948},
	title = {Extended convexity and uniqueness of minimizers for interaction energies}}

@article{DOSW,
	author = {P. D. Dragnev and R. Orive and E. B. Saff and F. Wielonsky},
	journal = {Constructive Approximation},
	pages = {1-43},
	title = {Riesz Energy Problems with External Fields and Related Theory},
	volume = {57},
	year = {2023}}

@article{Shu_expli,
	author = {R. Shu},
	journal = {preprint, arXiv:2501.14666},
	title = {A family of explicit minimizers for interaction energies}}

@article{CV11,
	author = {L. A. Caffarelli and J. L. V\'{a}zquez},
	journal = {Discrete Contin. Dyn. Syst.},
	number = {4},
	pages = {1393-1404},
	title = {Asymptotic behaviour of a porous medium equation with fractional diffusion},
	volume = {29},
	year = {2011}}

@article{Frost,
	author = {O. Frostman},
	journal = {Ph.D. thesis, Facult´e des Sciences de Lund},
	title = {Potentiel d'equilibre et capacit\'{e} des ensembles},
	year = {1935}}

@article{BCLR13_2,
	author = {D. Balagu{\'e} and J. A. Carrillo and T. Laurent and G. Raoul},
	journal = {Physica D},
	pages = {5-25},
	title = {Nonlocal interactions by repulsive--attractive potentials: Radial ins/stability},
	volume = {260},
	year = {2013}}

@article{CH17,
	author = {J. A. Carrillo and Y. Huang},
	journal = {Kinetic and Related Models},
	number = {1},
	title = {Explicit equilibrium solutions for the aggregation equation with power-law potentials},
	volume = {10},
	year = {2017}}

@article{FM,
	author = {R. L. Frank and R. W. Matzke},
	journal = {Archive for Rational Mechanics and Analysis},
	number = {2},
	pages = {15},
	title = {Minimizers for an aggregation model with attactive-repulsive interaction},
	volume = {249},
	year = {2025}}

@article{DLM1,
	author = {C. Davies and T. Lim and R. J. McCann},
	journal = {SIAM J. Appl. Math.},
	number = {4},
	pages = {1520-1536},
	title = {Classifying minimum energy states for interacting particles: spherical shells},
	volume = {82},
	year = {2022}}

@article{Fra,
	author = {R. L. Frank},
	journal = {Nonlinear Analysis},
	pages = {112691},
	title = {Minimizers for a one-dimensional interaction energy},
	volume = {216},
	year = {2022}}

@article{SW21,
	author = {R. Shu and J. Wang},
	journal = {arXiv preprint arXiv:2109.11006},
	title = {The sharp {E}rd{\H{o}}s--{T}ur{\'a}n inequality}}

@article{Lop19,
	author = {O. Lopes},
	journal = {Commun. Pure Appl. Anal.},
	number = {5},
	pages = {2265-2282},
	title = {Uniqueness and radial symmetry of minimizers for a nonlocal variational problem},
	volume = {18},
	year = {2019}}

@article{SST15,
	author = {R. Simione and D. Slep\u{c}ev and I. Topaloglu},
	journal = {J. Stat. Phys.},
	number = {4},
	pages = {972-986},
	title = {Existence of ground states of nonlocal interaction energies},
	volume = {159},
	year = {2015}}

@article{CCP15,
	author = {J. A. Ca{\~{n}}izo and J. A. Carrillo and F. S. Patacchini},
	journal = {Arch. Ration. Mech. Anal.},
	number = {3},
	pages = {1197-1217},
	title = {Existence of compactly supported global minimisers for the interaction Existence of compactly supported global minimisers for the interaction energy},
	volume = {217},
	year = {2015}}

@article{CS21,
	author = {J. A. Carrillo and R. Shu},
	journal = {Calculus of Variations and Partial Differential Equations},
	number = {1},
	pages = {28},
	title = {From radial symmetry to fractal behavior of aggregation equilibria for repulsive-attractive potentials},
	volume = {62},
	year = {2023}}

\end{document}